\documentclass[12pt]{amsart}   
\pdfoutput=1 

\usepackage{graphicx,subfigure,color}
\usepackage{latexsym,amsmath,amsfonts,amscd,amsthm,epstopdf,lineno}
\usepackage[T1]{fontenc}
\usepackage[colorlinks=true,citecolor=blue,urlcolor=blue]{hyperref}

\usepackage{caption,here}
\usepackage{textcomp}
\usepackage{tikz}
\usetikzlibrary{decorations.pathmorphing,decorations.pathreplacing}
\usepackage{siunitx}
\usepackage{algorithm}  
\usepackage{enumitem}   

\usepackage{fullpage}

\newtheorem{theorem}{Theorem}[section]
\newtheorem{remark}{Remark}
\newtheorem{lemma}[theorem]{Lemma}
\newtheorem{corollary}[theorem]{Corollary}

\newcommand{\dt}{\Delta t}
\newcommand{\dx}{\Delta x}

\begin{document}
\title{Method of Lines Transpose:  A Fast Implicit Wave Propagator}


\author[M. Causley]{M. Causley}
\address{Department of Mathematics, Michigan State University, East Lansing, MI 48824}
\email{causleym@math.msu.edu}
\thanks{This work has been supported in part by AFOSR grants FA9550-11-1-0281, FA9550-12-1-0343 and FA9550-12-1-0455, NSF grant DMS-1115709, and MSU Foundation grant SPG-RG100059.}

\author[A. Christlieb]{A. Christlieb}

\author[Y. G\"{u}\c{c}l\"{u}]{Y. G\"{u}\c{c}l\"{u}}

\author[E. Wolf]{E. Wolf}

\subjclass{Primary 65N12, 65N40, 35L05}

\date{}

\begin{abstract}

%

As a follow up to \cite{Causley2013}, we provide a detailed description of the numerical implementation of an $O(N)$, A-stable, second order accurate solution of the wave equation, constructed from semi-discrete boundary value problems. We improve on the previous algorithm by replacing the Lax-type correction used in \cite{Causley2013}, which was necessary for convergence when $\Delta t < \Delta x/c$, with a more accurate spatial quadrature, which we prove is convergent. 

We also demonstrate that the resulting solver remains fast even in the case of unstructured meshes, can incorporate domain decomposition, and allows for the implementation of Dirichlet, Neumann, periodic  and outflow boundary conditions.

Building upon results for the 1d formulation, we utilize alternate direction implicit (ADI) splitting to achieve a fast $O(N)$ solver in higher spatial dimensions. Our solver is built upon line objects and, combined with the flexibility of the integral solver, allows us to solve problems on arbitrary spatial domains, by embedding the boundary in a regular Cartesian mesh. Our solver is designed to couple with particle codes, where scale separation is an issue. We therefore demonstrate the ability of our solver to take time steps well beyond that of the Courant-Friedrichs-Lewy (CFL) stability limit of explicit codes.

\bigskip
\noindent {\em Keywords}: Method of Lines Transpose, Tranverse
Method of Lines, Implicit Methods, Boundary Integral Methods,
Alternating Direction Implicit Methods, ADI schemes

\bigskip

\end{abstract}

\maketitle

\section{Introduction}

Numerical solutions to the wave equation have been an area of investigation for many decades. The wave equation is ubiquitous in the physical world, arising in acoustics, electromagnetics, and fluid dynamics. Our main interest is in electromagnetic wave propagation in plasmas, which are challenging due not only to the nonlinear coupling of the fields with ionized particles, but also to the disparate time scales introduced by the plasma frequency, which can vary by several orders of magnitude from the speed of light (normalized by an appropriate length scale).

Perhaps to most popular method for kinetic plasma simulations is the particle-in-cell (PIC) method \cite{Birdsall1976}. PIC simulations are comprised of two principle components. The first is the particle push, which relies on a Lagrangian description to move macro-particles in phase space, according to the Vlasov equation. The second component is the field solver, which couples the electromagnetic fields to the moving particles, according to Maxwell's equations. The particles are projected to the mesh, and the fields back to the particles using polynomial interpolation (the so-called PIC weighting).

The field solver is typically built with a finite difference time domain (FDTD), or finite volume time domain (FVTD) algorithm. Plasma problems generally require a solution which allows for time steps large compared to that dictated by the propagation speed. Since explicit schemes must obey the Courant-Friedrichs-Lewy (CFL) stability limit, $\Delta t \leq \Delta x/c $, where $\Delta t$ is the time step, and $\Delta x$ is the  width of a mesh cell. To overcome this time step restriction, we propose the use of an implicit method. Several works have implemented Maxwell solvers using an alternate direction implicit (ADI) formulation, \cite{Fornberg,Fornberga,Smithe2009}, resulting in A-stable schemes for which large time steps can be taken. Furthermore, higher orders of temporal accuracy can be achieved by Richardson extrapolation, provided the dispersion error is sufficiently small, so that the coarse and fine solutions are not out of phase. These field solvers are designed for the first order formulation of Maxwell's equations, so that the divergence free nature of the numerical solution can be ensured \cite{Smithe2009}. One drawback of FDTD and FVTD formulations of Maxwell's equations is the inability to accurately describe time dependent point sources, which are  generally not collocated on the mesh points. For this reason, we consider integral formulations, for which convolution with a delta function will give exact spatial resolution.

Historically, differential formulations of PDEs were desirable over boundary integral methods, which were computationally expensive. Over the past few decades, fast summation methods, such as the fast multipole method \cite{greengard1987fast}, and the tree-code algorithm \cite{Barnes1986} have been introduced, which reduces the computational complexity of computing particle interactions from $O(N^2)$ to $O(N)$, or $O(N\log N)$. These acceleration methods have been extended to a variety of kernels, and as a result, fast methods for computing boundary integral solutions of many PDEs have been developed, such as the wave equation \cite{Coifman1993, Alpert2000,Alpert2002,Li2006}, Helmholtz equation \cite{Cheng1999}, Poisson equation \cite{Lindsay2001}, and even the modified Helmholtz equation \cite{Gimbutas2002,Li2009,Cheng2006}. This work in this paper was initially motivated by the development of accurate particle methods for simulating bounded plasmas utilizing tree-codes \cite{Christlieb2004,Li2009,Lindsay2001}. However, rather than utilize a tree-code approach, we extend our methods to the multi-dimensional case using an ADI splitting.

Although the MOL$^T$ approach is more commonly used for parabolic problems \cite{Salazar2000, Jia2008}, it has been considered sparsely for the wave equation \cite{Schemann1998}. The work we present here is closely related to an independent set of works recently published \cite{Bruno2010,Lyon2010}, in which the wave equation is solved using a Fourier continuation-ADI (FC-ADI) method. First a semi-discrete boundary value problem (BVP) is formulated using MOL$^T$; next, an ADI splitting results in a series of one-dimensional BVPs, which are solved in turn by employing a Fourier continuation method, which scales as $O(N\log N)$. The FC-ADI method presented is sixth order in space and first order in time, but is subsequently lifted to fourth order in time using Richardson extrapolation, without pollution of the solution due to dispersion error \cite{Lyon2010}. Additionally, the FC-ADI algorithm is unconditionally stable, and can be extended to arbitrary geometries by using periodic extensions of the functions to impose boundary conditions at arbitrary (non-mesh point) locations.

We differ from the work of \cite{Bruno2010,Lyon2010} in that we expand the wave function $u$ using classical polynomial bases, rather than Fourier bases. This method of approach would normally lead to a convolution of complexity $O(N^2)$. However, we derive a fast convolution algorithm that utilizes the analytic properties of the one-dimensional Green's function, which is a decaying exponential. Due to the shift-invariance of the exponential, global convolutions can be decomposed into a local and far-away contribution, using exponential recursion. The result is an $O(N)$ fast convolution algorithm, which can impose boundary conditions at exact boundary values, without the need of the periodic extensions used in \cite{Bruno2010}. In fact, by not relying on FFTs, the convolution remains $O(N)$ even if the mesh spacing is irregular, or boundary points do not lie on the mesh. As shown below, we incorporate Dirichlet, Neumann, periodic and outflow boundary conditions in one spatial dimension with relative ease. Furthermore, due to the ADI splitting, the boundary integrals in higher dimensions are never explicitly formed. Instead, all boundary conditions are implemented by solving one-dimensional (two-point) boundary value problems along ADI lines, and the ADI sweeps couple the information, to implicitly construct the boundary integral. Thus, our algorithm remains $O(N)$ in higher dimensions.

The analytical study of our wave solver was presented in \cite{Causley2013}, where it was shown to be A-stable, and convergent to the wave equation with second order accuracy in space and time, provided that a Lax correction is included to correct for the spatial discretization error. We shall present a slight modification of this result in Section \ref{sec:Fast}, which removes the necessity of the Lax correction, without affecting stability. In short, the polynomial bases we use to approximate the wave function $u$ must be of degree $\geq 2$.

The rest of this paper is laid out as follows. In section \ref{sec:molt}, the semi-discrete boundary integral solution to the wave equation in one spatial dimension is presented. We show how boundary conditions, as well as transmission conditions to a finite domain can be applied. The latter of these is useful in particular for domain decomposition, and thus parallelization of the algorithm, as well as in deriving outflow boundary conditions. In section \ref{sec:Fast}, we present the spatial discretization of the boundary integral solution, and describe how a fast $O(N)$ algorithm can be designed for the fully discrete solution. To avoid the Lax correction used in \cite{Causley2013}, a compact form of Simpson's rule is used for the spatial quadrature. We present the case of a regular grid, as well as unstructured mesh. Next, we incorporate the ADI splitting to solve the wave equation in higher dimensions. In section \ref{sec:ADI}, we propose a splitting in which all spatial derivatives are brought to one side of the equation, so that intermediate variables can no longer be strictly interpreted as intermediate time levels of the solution. As such, we detail the consistent incorporation of boundary integral terms, and the inclusion of sources. In section \ref{sec:results} we demonstrate our solution to be fast, second order accurate, and A-stable, even on non-rectangular domains. Finally, we conclude with several remarks in section \ref{sec:conclusion}.

\section{Semi-discrete solution using MOL$^T$}
\label{sec:molt}

The numerical algorithm is based on the initial boundary value problem for the wave equation in one spatial dimension
\begin{align}
	\label{eqn:wave}
	\frac{\partial^2 u}{\partial x^2}  - \frac{1}{c^2}\frac{\partial^2 u}{\partial t^2}&= -S(x,t), \quad  a< x < b, \quad t>0		\\
	u(x,0)	&= f(x), \quad a < x < b \nonumber														\\
	u_t(x,0)	&= g(x), \quad a < x < b, \nonumber
\end{align}
where $c$ is the propagation speed. The problem is well-posed once consistent boundary conditions are appended. We will consider below several important cases:  Dirichlet
\begin{equation}
	\label{eqn:Dirichlet}
	u(a,t) = U_L(t), \quad u(b,t) = U_R(t),
\end{equation}
Neumann
\begin{equation}
	\label{eqn:Neumann}
	u_x(a,t) = V_L(t), \quad u_x(b,t) = V_R(t),
\end{equation}
and periodic boundary conditions
\begin{equation}
	\label{eqn:Periodic}
	u(a,t) = u(b,t), \quad u_x(a,t) = u_x(b,t).
\end{equation}
We also address outflow boundary conditions in 1D, which can be formulated exactly in terms of local differential operators 
\begin{equation}
	\label{eqn:outflow}
	u_t(a,t) -c u_x(a,t) =0, \quad  u_t(b,t)+c u_x(b,t) = 0.
\end{equation}

\subsection{Time discretization}
We begin by discretizing $u_{tt}$ using the time-centered finite difference approximation
\[
	u_{tt}^n = \frac{u^{n+1}-2u^n+u^{n-1}}{\Delta t^2} - \frac{\Delta t^2}{12} u_{tttt}(x,\eta), \quad \eta \in [t_{n-1},t_{n+1}].
\]
Now, in order to obtain an implicit method, we require that the Laplacian be taken at time $t_{n+1}$. However, since the finite difference stencil is centered, we consider a symmetric 3-point averaging for the Laplacian
\[
	\frac{\partial^2}{\partial x^2} u^n = \frac{\partial^2}{\partial x^2} \left(  u^n+\frac{u^{n+1}-2u^{n}+u^{n-1}}{\beta^2} \right)-\frac{\dt^2}{\beta^2} u_{ttxx}(x,\eta),
\]
where $\beta>0$ is a parameter, and which produces the semi-discrete equation
\begin{equation}
	\label{eqn:betaTC} 
	\left(\frac{\partial^2}{\partial x^2} -\frac{\beta^2}{(c\Delta t)^2} \right)\left(u^n + \frac{u^{n+1}-2u^n+u^{n-1}}{\beta^2} \right) = - \frac{\beta^2}{(c\Delta t)^2}u^n -S(x,t_n).
\end{equation}
This method will result in a purely dispersive scheme, in that the numerical solution does not introduce numerical diffusion. Here $\beta>0$ is a free parameter which can be chosen to reduce the truncation error, but still produce an A-stable scheme. In fact, we have the following
\begin{lemma}
	Equation \eqref{eqn:betaTC} is unconditionally stable for $0<\beta \leq 2$. For $\beta = \sqrt{12}$, the scheme is fourth order accurate, but conditionally stable. The optimal choice which maintains A-stability and minimizes the error is $\beta =2$.
\end{lemma}
\begin{proof}
Consider first the stability of the scheme \eqref{eqn:betaTC} with $S = 0$. In order to ensure that the approximation remains finite we take $\beta \neq 0$, and without loss of generality, we assume $\beta>0$. Let $u^n(x) = \rho^n e^{ikx} u^{0}$, and define $\omega = kc$. Substitution and cancellation of the common terms yields the Von-Neumann polynomial
\begin{align*}
	&\left(-\left(\frac{\omega}{c}\right)^2 -\frac{\beta^2}{(c\Delta t)^2}\right)\left(\rho +\frac{\rho^2-2\rho+1}{\beta^2}\right) =-\frac{\beta^2}{(c\Delta t)^2}\rho\quad \implies \\
	&\rho^2-\left(2-\frac{(\beta\omega\Delta t)^2}{\beta^2+(\omega\Delta t)^2}\right)\rho+1 = 0
\end{align*}

Stability follows from the roots of this quadratic polynomial satisfying $|\rho|\leq 1$, and applying the Schur criterion leads to
\[
	\left |2-\frac{(\beta\omega\Delta t)^2}{\beta^2+(\omega\Delta t)^2} \right | \leq 2.
\]
A-stability follows from this inequality being satisfied for all values of $\omega$, and so we find $0<\beta \leq 2$.

Next, we look at the truncation error of the method. Upon substituting the exact solution $u$ into the semi-discrete equation \eqref{eqn:betaTC} and setting $S = 0$, we find the global truncation error
\begin{align*}
	\tau := &\left(\frac{\partial^2}{\partial x^2} -\frac{\beta^2}{(c\Delta t)^2} \right)\left(u^n + \frac{u^{n+1}-2u^n+u^{n-1}}{\beta^2} \right) + \frac{\beta^2}{(c\Delta t)^2}u^n \\
		= &\frac{\Delta t^2}{\beta^2}u_{ttxx} - \frac{\Delta t^2}{12c^2}u_{tttt}+O(\Delta t^4) \\
		= &\frac{(c\Delta t)^2}{12\beta^2}u_{xxxx}\left(12-\beta^2\right) +O(\Delta t^4).
\end{align*}
Thus, the truncation error will be 4th order if we choose $\beta = \sqrt{12}$. However, this is not in the range of A-stability. Observe that as $\beta$ increase, the error constant in the truncation error decreases. Thus, we define the value $\beta=2$ as optimal in the sense that it produces the A-stable scheme with the smallest discretization error.
\end{proof}

				\subsection{Integral solution and update equation}
The differential operator which appears in the semi-discrete equation \eqref{eqn:betaTC} is the modified Helmholtz operator, which we define by
\begin{equation}
	\label{eqn:HelmholtzL}
	\mathcal{L}_\beta[w](x) := \left(\frac{1}{\alpha^2}\frac{\partial^2}{\partial x^2} - 1\right)w(x), \quad \alpha = \frac{\beta}{c\dt}.
\end{equation}
A modified Helmoltz equaiton of the form
\begin{equation}
	\label{eqn:Helmholtz_Equation}
	\mathcal{L}_\beta[w](x) = -u(x), \quad a\leq x \leq b
\end{equation}
can be formally solved by inverting the Helmholtz operator using the Green's function (details can be found in \cite{Causley2013}). We define convolution with this Green's function by  the integral operator
\begin{equation}
	\label{eqn:Iu}
	I[u](x) := \alpha\int_a^b u(y)e^{-\alpha|x-y|}dy, \quad a\leq x \leq b,
\end{equation}
so that the integral solution is given by $\mathcal{L}_\beta^{-1}[u](x) = \frac{1}{2}w(x)$, where
\begin{equation}
	\label{eqn:L_Inverse}
	w(x):= \underbrace{I[u](x)}_{\text{Particular Solution}}+ \underbrace{\vphantom{I[u](x)} A e^{-\alpha(x-a)} + B e^{-\alpha(b-x)}}_{\text{Homogeneous Solution}}.
\end{equation}
The coefficients $A$ and $B$ of the homogeneous solution are determined by applying boundary conditions. The homogeneous solution \eqref{eqn:L_Inverse} can also be used as a means to incorporate information about $u$ from outside the domain $\Omega = [a,b]$. In this case, the coefficients $A$ an $B$ are enforcing transmission conditions. That is, suppose that the support of $u$ is $\Omega = \mathbb{R}$, so that we have the free space solution to the modified Helmholtz equation
\[
	w(x) = \alpha \int_{-\infty}^\infty u^n(y)e^{-\alpha|x-y|}dy.
\]
Now, we shall only ever evaluate $w(x)$ for $a\leq x \leq b$. Then, 
\begin{align}
	w(x) =& I[u](x) + \alpha\int_{-\infty}^a u^n(y)e^{-\alpha(x-y)}dy + \alpha \int_b^\infty u^n(y) e^{-\alpha(y-x)}dy \nonumber \\
		\label{eqn:Transmission}
		=& I[u](x) + A^ne^{-\alpha(x-a)} + B^ne^{-\alpha(b-x)},
\end{align}
where the homogeneous coefficients are
\begin{align}
		A^n =& \alpha\int_{-\infty}^a u^n(y)e^{-\alpha(a-y)}dy \\
		B^n =& \alpha \int_b^\infty u^n(y) e^{-\alpha(y-b)}dy,
\end{align}
and do not depend on $x$.

\begin{remark}
	The distinction to be made here is that the integral solution \eqref{eqn:L_Inverse} is valid when the domain $[a,b]$ is the full support of the function $u$, and the homogeneous solution is used to apply boundary conditions. This is in contrast to the integral solution \eqref{eqn:Transmission}, for which the support of $u$ extends beyond $[a,b]$, but we shall only ever evaluate $I[u](x)$ for $a\leq x \leq b$. This latter result will be used below to derive outflow boundary conditions, and, along with the fast convolution algorithm of Section \ref{sec:Fast}, to build an efficient domain decomposition algorithm.
\end{remark}

We now make a few key observations about the particular solution, which will be used extensively in the ensuing discussion. The integral operator can be decomposed into a left and right oriented integral, split at $y=x$ so that 
\begin{equation}
	\label{eqn:ILR}
	I[u](x) = I^L[u](x)+I^R[u](x),
\end{equation}
where
\[
	I^L[u](x) = \alpha \int_a^x e^{-\alpha(x-y)}u(y)dy, \quad  I^R[u](x)= \alpha\int_{x}^b e^{-\alpha(y-x)}u(y)dy.
\]
We may interpret $I^L$ and $I^R$ as the "characteristics" of $I$, as they independently satisfy first order "initial value problems"
\begin{align}
	(I^L)'(y) + \alpha I^L(y) = +\alpha u(y), \quad a<y<x,	\quad I^L(a) = 0 \\
	(I^R)'(y) - \alpha I^R(y) = -\alpha u(y), \quad x<y<b,	\quad I^R(b) = 0,
\end{align}
where the prime denotes spatial differentiation. The solutions to these equations can be found by the integrating factor method. Integrating $I^L$ from $x-\delta_L$ to $x$, and $I^R$ from $x$ to $x+\delta_R$, we find
\begin{align}
	\label{eqn:IL_def}
	I^L[u](x) &= I^L[u](x-\delta_L) e^{-\alpha \delta_L}+ J^L[u](x), \quad J^L[u](x):= \alpha \int_{0}^{\delta_L} u(x-y) e^{-\alpha y}dy \\
	\label{eqn:IR_def}
	I^R[u](x) &= I^R[u](x+\delta_R) e^{-\alpha \delta_R}+ J^R[u](x), \quad J^R[u](x):= \alpha \int_{0}^{\delta_R} u(x+y) e^{-\alpha y}dy.
\end{align}
The recursive updates \eqref{eqn:IL_def} and \eqref{eqn:IR_def} are exact (in space), and making $\delta_L$ and $\delta_R$ small (typically, $\delta_L =\delta_R = \Delta x$) effectively localizes the contribution of the integrals. We also observe that combining \eqref{eqn:IL_def} and \eqref{eqn:IR_def}, the total integral operator also satisfies a recursive definition
\begin{equation}
	\label{eqn:I_Rec}
	I[u](x) = e^{-\alpha \delta_L}I^L[u](x-\delta_L)+e^{-\alpha \delta_R}I^R[u](x+\delta_R) + \alpha \int_{-\delta_L}^{\delta_R} u(x+y) e^{-\alpha |y|}dy.
\end{equation}
In this expression, the remaining integral contains the "local" information of $I[u]$.

We now identify the semi-discrete form of the wave equation \eqref{eqn:betaTC} with the modified Helmholtz equation \eqref{eqn:Helmholtz_Equation} using
\[
	w(x) = 2\left(u^n(x)+\frac{u^{n+1}(x)-2u^n(x)+u^{n-1}(x)}{\beta^2}\right), \quad u(x) =\left(u^n(x)+\frac{1}{\alpha^2}S^n(x)\right)
\]
and upon inverting the Helmholtz operator and solving for $u^{n+1}$, find the update equation
\begin{equation}
	\label{eqn:Integral_Solution_Full}
	u^{n+1} = -(\beta^2-2)u^n-u^{n-1}+\frac{\beta^2}{2}\left(I\left[u^n+\frac{1}{\alpha^2}S^n\right]+A^n e^{-\alpha(x-a)}+B^ne^{-\alpha(b-x)}\right).
\end{equation}

\subsection{Consistency of the Particular Solution}   \label{Consistency_Particular_Solution}
In \cite{Causley2013}, quadrature was performed on the convolution integral \eqref{eqn:Iu} using the midpoint and trapezoidal rules. It was shown that a Lax-type correction was required to achieve a consistent numerical scheme for the wave equation, due to coupling between the spatial quadrature error, and the temporal truncation error. This is apparently a difficulty that is intrinsic to using the MOL$^T$ to formulate second order boundary value problems. In particular, a discussion of this same phenomenon, for parabolic problems, is carried out by Bruno and Lyon in Section 4 of \cite{Bruno2010}. It is indicated therein that the numerical solution, which is found using Fourier continuation methods, diverges upon letting $\Delta t \to 0$ when the spatial parameter $h$ held fixed. Bruno and Lyon \cite{Bruno2010} mitigate this difficulty by correcting the solution with a finite difference solver.

We will now show how such difficulties can be avoided. Specifically if $u(x+ y)$ in equation \eqref{eqn:I_Rec} is approximated with a polynomial of degree $p\geq 2$ and integrated analytically, the resulting fully discrete update will be stable and convergent. In the computational algorithm, however the recursive updates \eqref{eqn:IL_def} and \eqref{eqn:IR_def} will be used. Since they are computed separately, we must insist that each of these updates uses the same polynomial, so as to avoid jumps in higher derivatives, which (although they will be $O(\Delta x^p)$ for some $p$), will cause numerical instabilities if $\Delta t \ll \Delta x$.

Our reasoning is as follows. If from the update equation \eqref{eqn:Integral_Solution_Full} we solve for the finite difference approximation of $\frac{1}{c^2}u_{tt}$ and neglect the homogeneous solution, we find
\[
	\frac{u^{n+1}-2u^{n}+u^{n-1}}{(c\Delta t)^2} = \frac{\alpha^2}{2}\left(I\left[u^n+\frac{1}{\alpha^2}S^n \right]-2u^n\right),
\]
where again $\alpha = \beta/(c\Delta t)$. Thus, letting $\Delta t \to 0$ on the left hand side is equivalent to $\alpha \to \infty$ on the right hand side, and in order to recover the wave equation \eqref{eqn:wave}, the following limit must hold
\[
	\lim_{\alpha \to \infty} \frac{\alpha^2}{2}\left(I\left[u^n+\frac{1}{\alpha^2}S^n \right]-2u^n\right) = u^n_{xx} + S^n,
\]
for $a<x<b$. We justify the neglect of the homogeneous solution in \eqref{eqn:Integral_Solution_Full} because, as $\alpha \to \infty$, both exponentials converge to zero for any $x \neq a,b$. Showing that this limit holds for the semi-discrete solution (i.e. when $u$ is continuous in space) follows from integrating analytically the Taylor expansion of $u(x+y)$ against the Green's function analytically.

We will find need of the following
\begin{lemma}
\label{Exp_Int}
	For integers $m\geq 0$ and real $\nu>0$,
	\begin{align}
		E_{m}(\nu) :=  \nu\int_0^1 \frac{z^m}{m!} e^{-\nu z}dz = \frac{1}{\nu^m} \left(1-e^{-\nu }P_m(\nu)\right)
	\end{align}
	where
	\[
		P_m(\nu) = \sum_{\ell=0}^m\frac{\nu^\ell}{\ell!}
	\]
	is the Taylor series expansion of order $m$ of $e^{\nu}$.
\end{lemma}
\begin{proof}
The proof of this lemma follows from iterated use of integration by parts.
\end{proof}

We are now prepared to prove the following
\begin{theorem}
\label{Consistency_thm}
Let $u^n(x)$ be a semi-discrete solution given by \eqref{eqn:Integral_Solution_Full}. Then, $u^n(x)$ converges to $u(x,t_n)$ satisfying the wave equation \eqref{eqn:wave}, iff for each $x \in (a,b)$,
\begin{equation}
	\label{eqn:Consistency}
	\lim_{\alpha \to \infty} \alpha^2\left(I[u^n]-2u^n\right) =2 \frac{\partial^2}{\partial x^2}u^n.
\end{equation}
\end{theorem}

\begin{remark}
	This theorem is stated without the inclusion of sources for simplicity. However, all results apply, provided that $S \in L^1([a,b])$.
\end{remark}

\begin{proof}
Upon performing a Taylor expansion of $u^n(x\pm y)$ appearing in the integrals of the recurrence relations \eqref{eqn:IL_def} and \eqref{eqn:IR_def} (with $\delta_L = \delta_R = \delta$), and evaluating with the aid of Lemma \ref{Exp_Int}, we find
\begin{align*}
	\alpha \int_{0}^\delta u^n(x\pm y)e^{-\alpha y}dy &=  \nu \int_{0}^1 \left(\sum_{m=0}^3 \frac{z^m}{m!}\left(\pm \frac{\nu}{\alpha}\right)^m\frac{\partial^m u^n(x)}{\partial x^m}+O(z^4\delta^4) \right)e^{-\nu z}dz \\
		&= \sum_{m=0}^3 \left(\frac{\pm 1}{\alpha}\right)^m \frac{\partial^m u^n(x)}{\partial x^m}\left(1-e^{-\nu}P_m(\nu)\right) + O\left(\frac{1}{\alpha^4}\right)
\end{align*}
where $\nu = \alpha\delta$. When these expansions are summed, all odd powers of $\alpha$ cancel, and the recursive form of the integral \eqref{eqn:I_Rec} becomes
\[
	I(x) = 2u + \frac{2}{\alpha^2}\frac{\partial^2u}{\partial x^2}+ e^{-\nu}\left(I^L(x-\delta)+I^R(x+\delta)-2u - \frac{2}{\alpha^2}\frac{\partial^2u}{\partial x^2}\left(1+\nu +\frac{\nu^2}{2}\right)\right) + O\left(\frac{1}{\alpha^4}\right)
\]
and so, since $\alpha^2\left(I[u^n]-2u^n\right) =2 u^n_{xx} +O\left(e^{-\alpha \delta}+\frac{1}{\alpha^2}\right)$, taking the limit produces precisely the desired result \eqref{eqn:Consistency}.
\end{proof}
\begin{remark}
The proof of this theorem, while straightforward, underscores the necessary conditions for a fully discrete algorithm to recover a convergent approximation of the wave equation. That is, in order to recover the consistency condition \eqref{eqn:Consistency}, the first three terms of the Taylor expansion, produced by integrating polynomials up to degree 2 against an exponential, must be computed exactly. We therefore demand that the spatial quadrature used to discretize the integral \eqref{eqn:I_Rec} be exact up to degree $p\geq 2$.
\end{remark}

Consider now a fully discrete solution $u^n_j = u^n(x_j)$, which is obtained by discretizing the integral \eqref{eqn:I_Rec}, to order $p$. Recall that when constructing the numerical algorithm, we will use the recurrence updates for $I^L$ \eqref{eqn:IL_def} and $I^R$ \eqref{eqn:IR_def} to perform the updates. Therefore, to ensure the resulting polynomial approximation for the integral of equation \eqref{eqn:I_Rec} is continuous up to order $p$, we must use the same polynomial to perform quadrature on the integrals in equations \eqref{eqn:IL_def} and \eqref{eqn:IR_def}. Otherwise, the jump in higher derivatives may be of the form $(C_1 \Delta x^p) z$, and upon integrating against the Green's function we would find
\[
	\alpha^2\left(I[u^n]-2u^n\right) = C_1 \Delta x^p \alpha +2 u^n_{xx} +O\left(e^{-\alpha \delta}+\frac{1}{\alpha^2}\right),
\]
which, for fixed $\Delta x$, diverges as $\alpha \to \infty$. We formalize this in the following

\begin{corollary}
\label{Consistency_cor}
Let $u^n_j$ be a fully discrete solution, formed by replacing $u^n(x_j\pm y)$ with polynomials of degree $p\geq 2$ appearing in the integrals of equations \eqref{eqn:IL_def} and \eqref{eqn:IR_def}. Then, the discrete analog of equation \eqref{eqn:Consistency}, namely
\begin{equation}
	\label{eqn:Consistency_Disc}
	\lim_{\alpha \to \infty} \alpha^2\left(I[u^n](x_j)-2u^n_j\right) =2 \frac{\partial^2}{\partial x^2}u^n_{j}+O(\Delta x^p)
\end{equation}
is satisfied for each $x_j \in (a,b)$, iff the resulting polynomial approximation for the integral of equation \eqref{eqn:I_Rec} is continuous up to order $p$.
\end{corollary}

In Section \ref{sec:Fast}, we derive the update equations for the fully discrete numerical scheme, with $p=2$. This removes the necessity for the Lax correction used in \cite{Causley2013}.

				\subsection{Boundary conditions}
Before turning our attention to the discretization of the update equation \eqref{eqn:Integral_Solution_Full}, we will discuss the homogeneous part of the integral solution \eqref{eqn:L_Inverse}, which is used to enforce the boundary conditions.

Let us begin with Dirichlet boundary conditions \eqref{eqn:Dirichlet}. Evaluating the semi-discrete solution \eqref{eqn:Integral_Solution_Full} at $x = a$ and $b$, we find
\begin{align*}
	U_L(t_{n+1})=-(\beta^2-2)U_L(t_n)-U_L(t_{n-1}) + \frac{\beta^2}{2} \left(I\left[u^n+\frac{1}{\alpha^2}S^n \right](a)+A  + B e^{-\alpha(b-a)}\right), \\
	U_R(t_{n+1})=-(\beta^2-2)U_R(t_n)-U_R(t_{n-1}) + \frac{\beta^2}{2} \left(I\left[u^n+\frac{1}{\alpha^2}S^n \right](b)+A e^{-\alpha(b-a)} + B\right),
\end{align*}
which, after solving for the unknown coefficients can be written as
\begin{align*}
		A^n +\mu	B^n	&= -w_a^D, \\
	\mu	A^n +	B^n	&=  -w_b^D,
\end{align*}
with
\begin{align*}
	w_a^D &= I\left[u^n+\frac{1}{\alpha^2}S^n \right](a) - \frac{2}{\beta^2}\left(U_L(t^{n+1})+(\beta^2-2)U_L(t^{n}) +U_L(t^{n-1}) \right), \\
	w_b^D &= I\left[u^n+\frac{1}{\alpha^2}S^n \right](b) - \frac{2}{\beta^2}\left(U_R(t^{n+1})+(\beta^2-2)U_R(t^{n})+U_R(t^{n-1}) \right),
\end{align*}
and $\mu = e^{-\alpha(b-a)}$. Homogeneous boundary conditions are recovered upon setting $U_L(t) = U_R(t) = 0$. Solving the resulting linear system for the unknowns $A^n$ and $B^n$ gives
\begin{align}
	\label{eqn:wh_Dirichlet}
	A = -\left(\frac{w_a^D - \mu w_b^D }{1-\mu^2}\right), \quad B = -  \left(\frac{w_b^D - \mu w_a^D}{1-\mu^2}\right).
\end{align}
Before considering Neumann conditions, first observe that all dependence on $x$ in the integral solution \eqref{eqn:L_Inverse} is on the Green's function, which is a simple exponential function. Using this, we obtain the following identities
\begin{align}
	\label{eqn:DtN}
	I'(a) = \alpha I(a), \quad  I'(b) = -\alpha I(b).
\end{align}
Now, differentiating the semi-discrete solution \eqref{eqn:Integral_Solution_Full}, and applying the Neumann boundary conditions \eqref{eqn:Neumann} at $x=a$ and $b$ yields
\begin{align*}
	V_L(t_{n+1})=-(\beta^2-2)V_L(t_n)-V_L(t_{n-1}) + \alpha \frac{\beta^2}{2} \left(I\left[u^n+\frac{1}{\alpha^2}S^n \right](a)-A  + B e^{-\alpha(b-a)}\right), \\
	V_R(t_{n+1})=-(\beta^2-2)V_R(t_n)-V_R(t_{n-1}) + \alpha \frac{\beta^2}{2} \left( - I\left[u^n+\frac{1}{\alpha^2}S^n \right](b)-A e^{-\alpha(b-a)} + B\right),
\end{align*}
which, after solving for the unknown coefficients can be written as
\begin{align*}
		A^n - \mu	B^n	&= w_a^N, \\
	-\mu	A^n +	B^n	&= w_b^N,
\end{align*}
with
\begin{align*}
	w_a^N &= I\left[u^n+\frac{1}{\alpha^2}S^n \right](a) - \frac{2}{\alpha\beta^2}\left(V_L(t^{n+1})+(\beta^2-2)V_L(t^{n}) +V_L(t^{n-1}) \right), \\
	w_b^N &= I\left[u^n+\frac{1}{\alpha^2}S^n \right](b) + \frac{2}{\alpha\beta^2}\left(V_R(t^{n+1})+(\beta^2-2)V_R(t^{n})+V_R(t^{n-1})\right).
\end{align*}
Upon solving the linear system we obtain
\begin{align}
	\label{eqn:wh_Neumann}
	A = \left(\frac{w_a^N +\mu w_b^N}{1-\mu^2}\right), \quad B =  \left(\frac{w_b^N+\mu w_a^N}{1-\mu^2}\right).
\end{align}

Finally, we can also impose periodic boundary conditions, by assuming that 
\[
	u^n(b) = u^n(a), \quad u^n_x(a) = u^n_x(b), \quad n\geq 0.
\]
Enforcing this in the semi-discrete solution \eqref{eqn:Integral_Solution_Full} then yields
\begin{align*}
	I[u^n](a) + A + B\mu  &= I[u^n](b) + A\mu + B, \\
	\alpha \left(I[u^n](a) - A + B\mu\right)  &= \alpha\left(- I[u^n](b) - A\mu + B\right),
\end{align*}
where we have used the identity \eqref{eqn:DtN} applied to derivatives of $I$. Solving this linear system is accomplished quickly by dividing the second equation by $\alpha$, and either adding or subtracting it from the first equation, to produce
\begin{align}
	\label{eqn:wh_Periodic}
	A = \frac{I[u^n](b)}{1-\mu}, \quad B = \frac{I[u^n](a)}{1-\mu}.
\end{align}
\begin{remark}
	The cases of applying different boundary conditions at $x=a$ and $b$ are not considered here, but the details follow from an analogous procedure to that demonstrated above.
\end{remark}

				\subsection{Outflow boundary conditions}
When computing wave phenomena, whether we are interested in finite or infinite domains, it is often the case that we must restrict our attention to some smaller subdomain $\Omega$ of the problem, which does not include the physical boundaries. We say that $\Omega$ is the \textit{computational domain}, and that the boundary $\partial \Omega$ is the non-physical, or \textit{artificial boundary}. Under these circumstances, it is necessary to enforce an outflow, or non-reflecting boundary condition, which allows the wave to leave the computational domain, but not incur (non-physical) reflections at the artificial boundary.

For the one-dimensional wave equation \eqref{eqn:wave} the exact outflow boundary conditions \eqref{eqn:outflow} turn out to be local in space and time. We emphasize that this is only the case in one spatial dimension, but we shall utilize this fact to obtain an outflow boundary integral solution from the integral equation \eqref{eqn:Transmission}. We extend the support of our function to $(-\infty,\infty)$, and extend the definition of the outflow boundary conditions to the domains exterior to $[a,b]$
\begin{align}
	u_t+cu_x = 0, \quad x\geq b, \\
	u_t-cu_x = 0, \quad x\leq a.
\end{align}

Next, assume the initial conditions have some compact support; for simplicity we will take this support to be $\Omega_0 = [a,b]$. Then after a time $t=t_n$, the domain of dependence of $u^n(x)$ is $\Omega_t = [a-ct_n,b+c t_n]$, since the propagation speed is $c$. Now the free space solution \eqref{eqn:Transmission} becomes
\begin{align}
	w(x)	=& \alpha \int_{a-c t_n}^{b+c t_n} e^{-\alpha|x-y|} u^n(y) dy \nonumber \\
			=& I[u](x)+ A^ne^{-\alpha(x-a)} + B^ne^{-\alpha(b-x)}
\end{align}
with coefficients
\begin{align}
		\label{eqn:A_Out_Def}
	A^n &= \alpha\int_{a-ct_n}^{a} e^{-\alpha(a-y)} u^n(y) dy, \\
		\label{eqn:B_Out_Def}
	B^n &= \alpha\int_{b}^{b+ct_n} e^{-\alpha(y-b)} u^n(y) dy.
\end{align}

At first glance, these coefficients are not at all helpful, as they require computing integrals along spatial domains which not only are outside of the computational domain, but also grow linearly in time. However, we will now make use of the extended boundary conditions to turn these spatial integrals into time integrals, which exist at precisely the endpoints $x=a$ and $b$ respectively. Consider first $x>b$. By assumption, this region contains only right traveling waves, $u(x,t) = u(x-ct)$, and by tracing backward along a characteristic ray we find
\[
	u(b+y,t) = u\left(b,t-\frac{y}{c}\right), \quad y>0.
\]
Thus,
\begin{align*}
	B^n &= \alpha \int_{0}^{ct_n} e^{-\alpha y} u(b+y,t_n) dy \\
		&= \alpha c \int_{0}^{t_n} e^{-\alpha c s} u\left(b,t_n-s\right) ds
\end{align*}
and so $B^n$ is equivalently represented by a convolution in time, rather than space. Now, knowing the history of $u$ at $x=b$ is sufficient to impose outflow boundary conditions. Furthermore, we find in analog to equation \eqref{eqn:IR_def}, a temporal recurrence relation due to the exponential
\begin{align*}
	B^n	&= \alpha c \int_{0}^{\Delta t} e^{-\alpha c s} u\left(b,t_n-s\right) ds + e^{-\alpha c \Delta t}\left(\int_0^{t_{n-1}} e^{-\alpha c s} u\left(b,t_{n-1}-s\right) ds \right) \\
		&= \beta \int_{0}^{1} e^{-\beta z} u\left(b,t_n-z\Delta t\right) dz +e^{-\beta} B^{n-1},
\end{align*}
where $\beta = \alpha c \Delta t$, by definition \eqref{eqn:HelmholtzL}. Thus, the coefficient $B^n$, which imposes an outflow boundary condition at $x=b$, can be computed locally in both time and space. To maintain second order accuracy, we fit $u$ with a quadratic interpolant
\[
	u(b,t_n-z\Delta t) \approx p(z) = u^n(b) -\frac{z}{2}\left(u^{n+1}(b)-u^{n-1}(b)\right) + \frac{z^2}{2}\left(u^{n+1}(b)-2u^n(b)+u^{n-1}(b)\right)
\]
and integrate the expression analytically using Lemma \ref{Exp_Int} to arrive at
\begin{equation}
	\label{eqn:Outflow_Update_Intermediate}
	B^n = e^{-\beta}B^{n-1} +\gamma_0 u^{n+1}(b)+ \gamma_1 u^n(b)+ \gamma_2 u^{n-1}(b)
\end{equation}
where
\begin{align*}
	\gamma_0 =& \frac{E_2(\beta)-E_1(\beta) }{2} = \frac{(1-e^{-\beta})}{\beta^2}-\frac{(1+e^{-\beta})}{2\beta} \\
	\gamma_1 =& E_0(\beta)-E_2(\beta)                 = -2\frac{(1-e^{-\beta})}{\beta^2}+ \frac{2}{\beta}e^{-\beta} +1  \\
	\gamma_2 =& \frac{E_2(\beta)+E_1(\beta)}{2} = \frac{(1-e^{-\beta})}{\beta^2}+\frac{(1-3e^{-\beta})}{2\beta} - e^{-\beta}.
\end{align*}
In this outflow update equation \eqref{eqn:Outflow_Update_Intermediate}, the quantities $u^{n+1}(b)$ and $B^n$ are both unknown. In order to determine these values, we must also evaluate the update equation for $u^{n+1}$ \eqref{eqn:Integral_Solution_Full} at $x=b$
\[
	u^{n+1}(b)+(\beta^2-2)u^{n}(b)+u^{n-1}(b) = \frac{\beta^2}{2}\left(I(b)+A^n \mu + B^n\right),	\quad \mu = e^{-\alpha(b-a)}.
\]
We now use these two equations to solve for $u^{n+1}(b)$, and eliminate it from the outflow update equation \eqref{eqn:Outflow_Update_Intermediate}, so that
\begin{equation}
	\label{eqn:Outflow_Update_B}
	-\Gamma_0 \mu A^n + (1-\Gamma_0)B^n = e^{-\beta}B^{n-1} +\Gamma_0 I(b) +\Gamma_1u^n(b)+ \Gamma_2 u^{n-1}(b)
\end{equation}
where
\[
	\Gamma_0 = \frac{\beta^2}{2}\gamma_0, \quad
	\Gamma_1 = \gamma_1-\gamma_0(\beta^2-2),	\quad
	\Gamma_2 = \gamma_2-\gamma_0
\]

\begin{remark}
	While this procedure could be avoided by omitting $u^{n+1}(b)$ in the interpolation stencil, it turns out to be necessary to obtain convergent outflow boundary conditions.
\end{remark}
Likewise, upon considering $x<a$, we find
\begin{equation}
	\label{eqn:Outflow_Update_A}
	(1-\Gamma_0) A^n -\Gamma_0 \mu B^n = e^{-\beta}A^{n-1} +\Gamma_0 I(a) +\Gamma_1u^n(a)+ \Gamma_2 u^{n-1}(a).
\end{equation}
Solving the resulting linear system produces
\begin{equation}
	\label{eqn:Outflow_Update}
	A^n = \frac{(1-\Gamma_0)w_a^{\text{Out}} +\mu \Gamma_0w_b^{\text{Out}}} {(1-\Gamma_0)^2-(\mu \Gamma_0)^2}, \quad
	B^n = \frac{(1-\Gamma_0)w_b^{\text{Out}} +\mu \Gamma_0w_a^{\text{Out}}} {(1-\Gamma_0)^2-(\mu \Gamma_0)^2},
\end{equation}
where
\begin{align*}
	w_a^{\text{Out}} &= e^{-\beta}A^{n-1} +\Gamma_0 I(a) +\Gamma_1u^n(a)+ \Gamma_2 u^{n-1}(a), \\
	w_b^{\text{Out}} &= e^{-\beta}B^{n-1} +\Gamma_0 I(b) +\Gamma_1u^n(b)+ \Gamma_2 u^{n-1}(b)
\end{align*}

\section{Treatment of point sources, and soft sources}
\label{sec:sources}

We now consider the inclusion of source terms. We are predominantly interested in the case where $S(x,t)$ consists of a large number of time dependent point sources. However, it is often the case that in electromagnetics problems, a soft source is prescribed to excite waves of a  prescribed frequency, or range of frequencies, within the domain. A soft source is so named because, although incident fields are generated at a prescribed fixed spatial location, no scattered fields are generated.

The implementation of a soft source $\sigma(t)$ at $x = x_s$ is accomplished by prescribing the source condition
\begin{equation}
	\label{eqn:condition}
	u(x_s,t) = \sigma(t).
\end{equation}
However, it can be shown that if we set
\begin{align}
	S(x,t)= \frac{2}{c}\sigma'(t) \delta(x-x_s)
\end{align}
and insert it into the wave equation \eqref{eqn:wave}, then the soft source condition \eqref{eqn:condition} is satisfied, and the solutions are equivalent. Thus, a soft source is nothing more than a point source, whose time-varying field is integrated by the wave equation.

Upon convolving this source term with the Green's function according to \eqref{eqn:Iu}, we find
\begin{align*}
	I\left[\frac{1}{\alpha^2}S\right](x)=& \frac{1}{\alpha}\int_a^b \left(\frac{2}{c}\sigma'(t_n)\delta(x-x_s) \right)e^{-\alpha|x-y|}dy \\
			=& \frac{2\Delta t}{\beta} \sigma'(t_n) e^{-\alpha|x-x_s|},
\end{align*}
where the definition of $\alpha = \beta/(c\Delta t)$ has been utilized.

\begin{remark}
	It is often the case that taking the analytical derivative $\sigma'(t_n)$ is to be avoided, for various reasons. In this case, any finite difference approximation which is of the desired order of accuracy can be substituted.
\end{remark}

Likewise for general point sources,
\[
	S(x,t) = \sum_{i} \tilde{\sigma}_i(t) \delta(x-x_i)
\]
the corresponding form of the source term is
\begin{equation}
	\label{eqn:Delta_Source}
	I\left[\frac{1}{\alpha^2}S\right](x) = \frac{c \Delta t}{\beta}\sum_i  \tilde{\sigma}_i(t_n) e^{-\alpha|x-x_i|}
\end{equation}
Therefore, it suffices to consider delta functions both for the implementation of soft sources, as well as including time dependent point sources.

\section{Spatial discretization}
\label{sec:Fast}

Having considered the homogeneous and source terms, we are now prepared to present a fully discrete numerical solution, defined by discretizing in space the update equation \eqref{eqn:Integral_Solution_Full}. Upon full discretization, computing $I_j \equiv I[u](x_j)$ according to \eqref{eqn:Iu} results in a dense matrix-vector product, which requires $O(N^2)$ operations to compute, where $N$ is the number of spatial grid points. However we will show that by utilizing the recurrence relations for $I^L$ \eqref{eqn:IL_def}, and $I^R$ \eqref{eqn:IR_def}, the particular solution $I = I^L + I^R$ can be formed in $O(N)$ operations, by means of fast convolution. This turns out to be true, even in the case of unstructured grids.

We first discretize the domain $[a,b]$ into $N$ evenly spaced subintervals, of width $\dx = (b-a)/N$, and define $x_j = a+j\dx$ for $0\leq j \leq N$. Appealing to the recursive definition of the integral operator \eqref{eqn:I_Rec} evaluated at $x_j$ and with $\delta_L=\delta_R = \Delta x$, we make a change of variables $y = z\Delta x$, and find
\begin{align}
	\label{eqn:I_Rec_Disc}
	I_{j}	= e^{-\nu} \left(I^L_{j-1}+I^R_{j+1}\right)	+ \nu \int_{-1}^{1} u(x_{j}+z\Delta x) e^{-\nu |z|}dz,
\end{align}
where $I_j = I[u](x_j)$, and  the discrete integration parameter is
\begin{equation}
	\label{eqn:nudef}
	\nu = \alpha\Delta x = \frac{\beta\Delta x}{c\Delta t}.
\end{equation}
We now replace $u(x_{j}+ z\Delta x)$ with a polynomial interpolant over the subinterval corresponding to $-1\leq z \leq 1$. We consider polynomials $p_{j}^{(m)}(z)$ of degree $m$, which use successively larger stencils involving the points $u_{j\pm k}$. The recursive updates for $I^L$ and $I^R$ become
\begin{align}
	\label{eqn:IL_Rec_Disc}
	I^L_{j}	=& e^{-\nu} I^L_{j-1}	+ J^L_j, \quad J^L_j = \nu \int_{0}^{1} u(x_{j}-z\Delta x) e^{-\nu z}dz, \\
	\label{eqn:IR_Rec_Disc}
	I^R_{j}	=& e^{-\nu} I^R_{j+1}+ J^R_j, \quad J^R_j = \nu \int_{0}^{1} u(x_{j}+z\Delta x) e^{-\nu z}dz.
\end{align}
These expressions are still exact, but the remaining integrals must be approximated using quadrature.

				\subsection{Compact Simpson's Rule}
Motivated by theorem \ref{Consistency_thm}, we now perform quadrature using second order polynomials $p^{(2)}_j$. To be precise, $u(x_j+z\Delta x)$ is approximated by
\begin{equation}
	\label{eqn:p2_def}
	p_j^{(2)}(z) = (1-z)u_{j} + z u_{j+1} + \left(\frac{z^2-z}{2}\right)\Delta x^2 u''(\xi_j), 
\end{equation}
where
\[
	\Delta x^2 u''(\xi_j) =
	\begin{cases}
		2u_{0}-5u_1+4u_{2}-u_3,				\quad &j = 0 \\
		u_{j+1}-2u_j+u_{j-1},				\quad &1 \leq j \leq N-1 \\
		2u_{N}-5u_{N-1}+4u_{N-2}-u_{N-3},		\quad &j = N,
	\end{cases}
\]
so that
\[
	u(x_j+z\Delta x) - p_j^{(2)}(z) = O(\Delta x^3).
\]
Now, we replace $u(x_j+z\Delta x)$ with $p_j^2(z)$ in $J^R_j$ of equation \eqref{eqn:IR_Rec_Disc}, and integrate the result analytically. By symmetry, the corresponding polynomial approximation for $u(x_j-z\Delta x)$ is made in $J^L_j$ of equation \eqref{eqn:IL_Rec_Disc}. Making use of Lemma \ref{Exp_Int}, we find
\begin{align}
	\label{eqn:IL_update}
	J^L_j  = Pu_j + Qu_{j-1}+R \Delta x^2 u''(\xi_j) \\
	\label{eqn:IR_update}
	J^R_j = Pu_j + Qu_{j+1}+R \Delta x^2 u''(\xi_j)
\end{align}
where $I^L_0 = 0$ and $I^R_N = 0$ by definition, and where the coefficients are
\begin{align}
	\label{eqn:ddef}
	d =& e^{-\nu} \\
	P =&  E_0(\nu)-E_1(\nu) = 1-\frac{1-d}{\nu} \nonumber \\
	Q =&  E_1(\nu)= -d+\frac{1-d}{\nu} \nonumber \\
	R =&  \frac{E_2(\nu)-E_1(\nu)}{2} = \frac{1-d}{\nu^2} - \frac{1+d}{2\nu}. \nonumber
\end{align}
We refer to this method as the compact Simpson's rule, to indicate that a 3-point stencil is used in each of the integrals in $J^L_j$ and $J^R_j$, which are only one computational cell in length. Notice that the only difference between $J^L_j$ and $J^R_j$ is the appearance of $u_{j\pm 1}$ multiplied by the coefficient $Q$, according to the direction of integration. We also emphasize that $I^L_j$ is updated for increasing $j$, and $I^R_j$ for decreasing $j$.

It now remains to show that the presented fully discrete scheme satisfies the consistency condition \eqref{eqn:Consistency_Disc}. First, we combine $I^L_j$ and $I^R_j$ to write the fully discrete update \eqref{eqn:I_Rec_Disc}, which upon simplification of the coefficients \eqref{eqn:ddef} becomes
\begin{align*}
	I_j	=& d(I^L_{j-1}+I^R_{j+1}) + 2Pu_j +Q(u_{j+1}+u_{j-1})+2R(u_{j+1}-2u_j+u_{j-1}) \\
		=& e^{-\nu}(I^L_{j-1}+I^R_{j+1}) + 2(1-e^{-\nu})u_j +\frac{2}{\nu^2}\left(1-\left(1+\nu+\frac{\nu^2}{2}\right)e^{-\nu}\right)(u_{j+1}-2u_j+u_{j-1}).
\end{align*}
If from this expression we solve for the quantity $I_j - 2u_j$ and multiply by $\alpha^2$ (recall, $\nu = \alpha \Delta x$) we have
\[
	\alpha^2(I_j-2u_j) =  2\frac{u_{j+1}-2u_j+u_{j-1}}{\Delta x^2}+O\left((\alpha \Delta x)^2e^{-\alpha \Delta x}\right), \quad \alpha\Delta x \gg 1
\]
and letting $\alpha \to \infty$ produces the second order finite difference approximation at $x_j$, confirming the consistency of the scheme.

\begin{remark}
	A similar approach could be used to formulate higher order quadrature formulae. In this case, several different stencils for $\Delta x^p u^{(p)}(\xi_j)$ would be required near the boundaries. Development and implementation of consistent quadrature rules is a topic of future work.
\end{remark}

				\subsection{Unstructured meshes}	\label{sec:unstructured}
We now consider unstructured meshes. More generally, define the partition of $[a,b]$ by the subintervals $[x_{j-1},x_j]$, where $x_0 = a$, $x_N = b$, and $x_{j-1}<x_j$. Likewise, the analogous discrete parameters become
\begin{align}
	h_j =& x_j - x_{j-1}, \quad \nu_j = \alpha h_j, \quad d_j = e^{-\nu_j} \\
	P_j =& 1-\frac{1-d_j}{\nu_j} \nonumber \\
	Q_j =& -d_j+\frac{1-d_j}{\nu_j} \nonumber \\
	R_j =& \frac{1}{2\nu_j^2}(2(1-d_j)-\nu_j(1+d_j)). \nonumber
\end{align}
Making use of the recurrence relations \eqref{eqn:IL_def} with $\delta_L = h_j$, and \eqref{eqn:IR_def} with $\delta_R = h_{j+1}$, we obtain
\begin{align}
	\label{eqn:ILj}
	I^L_j  =& d_j I^L_{j -1} + J^L_j, \quad J^L_j =\nu_j \int_0^1 u(x_j-z h_j) e^{-\nu_j z} dz \\
	\label{eqn:IRj}
	I^R_j =& d_{j+1} I^R_{j+1}+ J^R_j, \quad J^R_j = \nu_{j+1} \int_0^1 u(x_j+z h_{j+1}) e^{-\nu_{j+1} z} dz.
\end{align}
It only remains to construct the polynomial interpolants $p^{(2)}_j(z)$. However, the integration domain is still $0\leq z \leq 1$, and so $p^2_j(z)$ written in the form \eqref{eqn:p2_def} is still valid, as long as we replace $\Delta x$ with $h_j$, and modify stencils to accurately approximate the second derivative at $x_j$. Since this can be done in $O(N)$ operations, the scheme is still fast, and the discretization of $J^L_j$ and $J^R_j$ found in equations \eqref{eqn:IL_update} and \eqref{eqn:IR_update} respectively still apply, so long as the coefficients are replaced with their $j-$dependent counterparts.

\subsection{1D domain decomposition}
\label{sec:DD}
In this section we propose an efficient parallel algorithm for performing the 1D convolution \eqref{eqn:Iu} on a decomposed domain, with the appropriate boundary conditions. If $N$ is the total number of grid points in the domain, and $M$ the number of processors, then both the number of floating point operations and the amount of memory storage scale as $O(N\!+\!M)$. Once the convolution is computed, the solution update \eqref{eqn:Integral_Solution_Full} is completely local and does not require any further communication between the subdomains.

\begin{figure}[htbp]
	\begin{center}
	\includegraphics[width = \textwidth]{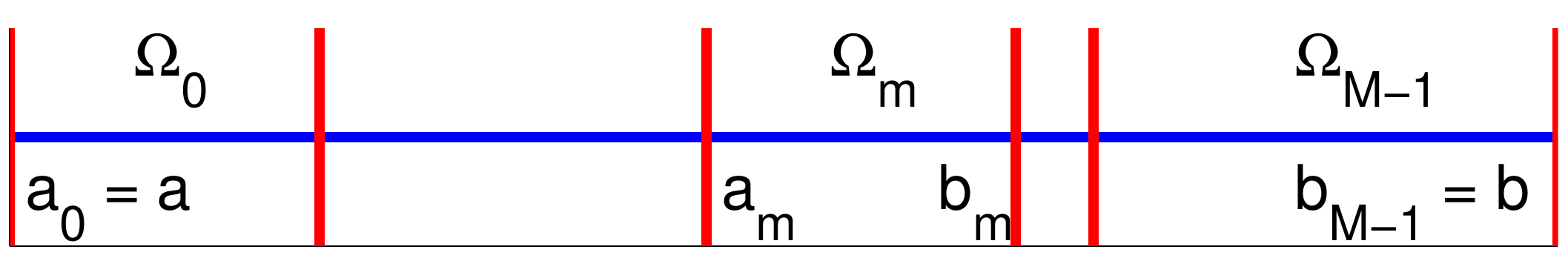}
	\caption{Decomposition of a 1D domain}
	\label{fig:Domain_Decomposition}
	\end{center}
\end{figure}

Consider a 1D domain $\Omega = [a,b]$, and decompose it into $M$ subdomains  $\Omega_m = [a_m,b_m]$ of generic sizes ($m = 0,1,\dots,M\!-\!1$),  as shown in Figure \ref{fig:Domain_Decomposition}. Appealing to the form of the integral solution with transmission coefficients \eqref{eqn:Transmission}, we see that for $x\in \Omega_m$, the integral solution can be written as
\[
	I[u](x) =  I_m[u](x) + A_m e^{-\alpha (x-a_m)} + B_m e^{-\alpha(b_m-x)},  \quad  x \in \Omega_m,
\]
where the "local particular solution" is
\[
	I_m[u](x) = \alpha \int\limits_{\Omega_m} u(y) e^{-\alpha |x-y|}dy, \quad
\]
and the local homogeneous coefficients are
\[
	A_m = \int_a^{a_m} e^{-\alpha(a_m-y)}u(y)dy, \quad B_m = \int_{b_m}^b e^{-\alpha(y-b_m)}u(y)dy.
\]
We point out here that if $m=0$ or $M-1$, then the coefficient $A_m$ or $B_m$ respectively are used to enforce boundary conditions, rather than transmission conditions. Now, suppose that $u^n(x)$ from the subdomain $\Omega_m$ is copied into local memory on machine $m$, and the integral solution \eqref{eqn:Integral_Solution_Full} is to be computed there, only for $x \in \Omega_m$. Then, the contributions from the total domain $\Omega$, including the boundary conditions, is determined by the scalar values $A_m$ and $B_m$. Additionally, the only information that must be passed to other domains $\Omega_k$ are the scalar values $I_m[u](a_m)$ and $I_m[u](b_m)$.

The idea that underlies the following algorithm is the analogy between a decomposed 1D domain, and a (perhaps non-uniform)  "coarse" mesh with $M$ cells: the interfaces between each subdomain $\Omega_m$ are interpreted as mesh nodes $X_m$ (with $m = 0,1,\dots,M$), where
\begin{equation*}
	X_0 = a_0, \qquad
	X_m = a_m = b_{m-1} \quad{\text{for $m=1,2,\dots M\!-\!1$}}, \qquad
	X_M = b_{M-1}.
\end{equation*}
Then, the global solution on the coarse mesh is decomposed into
\[
	I[u](X_m) = I^L[u](X_m) + I^R[u](X_m),
\]
with characteristics given by
\begin{align}
	I^L[u](X_m) &= e^{-\nu_j}I^L[u](X_{m-1}) + I_m[u](X_m), \\
	I^R[u](X_m) &= e^{-\nu_{m+1}}I^R[u](X_{m+1})+ I_{m+1}[u](X_m),
\end{align}
and where $\nu_m = \alpha(X_m-X_{m-1})$. Thus, global particular solution is computed only on the coarse mesh, and is comprised of the scalar values from the local particular solutions, evaluated at the endpoints of their respective domains. Once the global particular solution is computed, the boundary conditions can be applied, and the total integral solution will be known at each $X_m$. Finally, the local homogeneous coefficients $A_m$ and $B_m$ are obtained by 
\[
	A_m = I^L[u](X_m), \quad B_m = I^R[u](X_{m+1}).
\]
We now have sufficient background for describing algorithm 1 below.

\begin{figure}[htbp]
	\begin{center}
	\subfigure[Fine-coarse communication]{	\includegraphics[width = .47\textwidth]{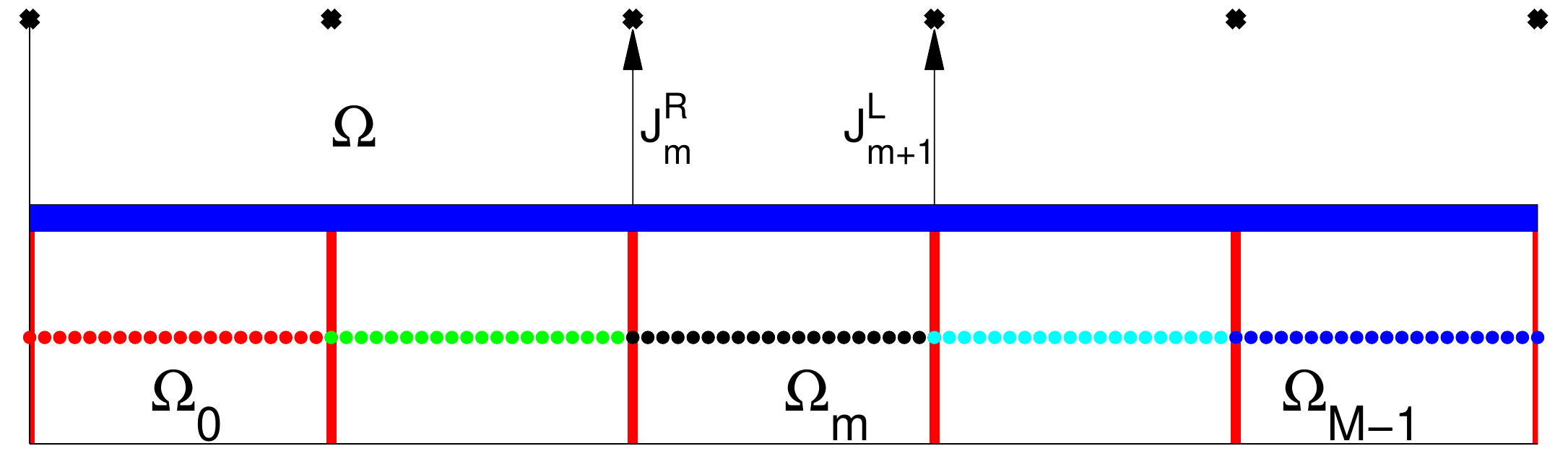}	\label{fig:Domain_Decomposition_2}}
	\subfigure[Coarse-fine communicatoin]{	\includegraphics[width = .47\textwidth]{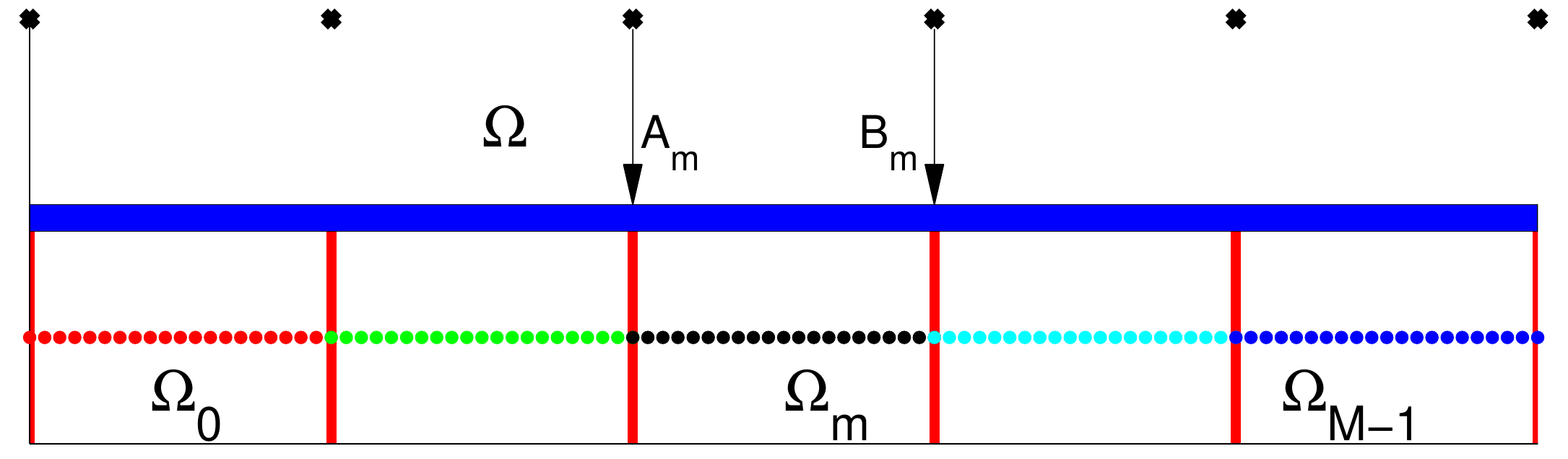}	\label{fig:Domain_Decomposition_3}}
	\caption{In algorithm 1, the "local particular solutions" $I_m[u](x)$ are found for each $\Omega_m$ and the values $J^L_{m+1} \equiv I_m[u](b_m)$ and $J^R_m\equiv I_m[u](a_m)$ are passed to the coarse grid (a). Once the "global particular solution" and boundary correction is computed on the coarse grid, the transmission coefficients $A_m$ and $B_m$ are passed to $\Omega_m$ (b).}
	\end{center}
\end{figure}

\begin{algorithm}[!htp]
\caption{Fast 1D domain decomposition}

\begin{enumerate}[itemsep=2mm]

\item On each subdomain $\Omega_m = [a_m,b_m]$, compute the ``local particular 
	solution'' 
	\[
		I_m[u^n](x) = \alpha \int_{a_m}^{b_m} e^{-\alpha |x-y|} u^n(y) dy,         \qquad x\in\Omega_m
	\]
	by means of the fast convolution algorithm, i.e. equations \eqref{eqn:IL_Rec_Disc} and \eqref{eqn:IR_Rec_Disc}; 

\item Pass the local values $J^L_{m+1} \equiv I_m[u](b_m)$ and $J^R_m \equiv I_m[u](a_m)$ to the coarse grid (\emph{communication});

\item Compute the contributions from the left and right characteristics on the coarse grid, by means of the recursive relations
	\begin{alignat*}{3}
	& I^L_0 = 0, \qquad 
	&				& I^L_m = e^{-\nu_{m-1}}I^L_{m-1}  + J^L_{m} \quad
	&				& \text{for $m = 1,2,\dots,M$, and} \\
	& I^R_{M} = 0, \qquad
	&				& I^R_m = e^{-\nu_{m}}I^R_{m+1} + J^R_{m} \quad
	&				& \text{for $m = M\!-\!1,M\!-\!2,\dots,0$},
	\end{alignat*}
	where $\nu_j = \alpha(b_{j}-a_j)$ for $j=0,1,\dots,M\!-\!1$;

\item Using the "global particular solution" at the endpoints $I[u](a)=I^R_0$ and $I[u](b)=I^L_{M-1}$, find the global coefficients $A$ and $B$ in accordance with the update equation \eqref{eqn:Integral_Solution_Full}, by imposing the required boundary conditions at $x = a$ and $b$;

\item Compute the local coefficients to contain the global boundary data:
	\begin{alignat*}{3}
	& A_0 = A, \qquad 
	&				& A_m =  I^L_m + A_{m-1} e^{-\nu_{m-1}} \quad
	&				& \text{for $m = 1,2,\dots,M\!-\!1$, and} \\
	& B_{M-1} = B, \qquad
	&				& B_m = I^R_m + B_{m+1} e^{-\nu_{m}} \quad
	&				& \text{for $m = M\!-\!2,M\!-\!3,\dots,0$};
	\end{alignat*}

\item Send $A_m$ and $B_m$ to the corresponding subdomain $\Omega_m$ (\emph{communication});

\item On each subdomain $\Omega_m = [a_m,b_m]$, construct the local integral solution, by summing the particular and homogeneous solutions
      \[ w_{m}[u](x) = I_m[u](x) + A_m e^{-\alpha(x-a_m)} + B_m e^{-\alpha(b_m-x)},          \qquad x\in\Omega_m. \]

\end{enumerate}
\end{algorithm}

\section{Higher spatial dimensions: ADI splitting }
\label{sec:ADI}

The full numerical algorithm in higher dimensions is based on the initial boundary value problem
\begin{align}
	\nabla^2 u - \frac{1}{c^2}\frac{\partial^2 u}{\partial t^2} &= -S(r,t), \quad  r \in \Omega, \quad t>0 \label{eqn:prob2} \\
	u(r,0) &= f(r), \quad r \in \Omega \nonumber \\
	u_t(r,0) &= g(r), \quad r \in \Omega \nonumber \\
	u(r,t) &= h(r,t), \quad r \in \partial\Omega, \quad t>0. \nonumber
\end{align}
We again utilize MOL$^T$ to perform the temporal discretization. Next we employ ADI splitting to the modified Helmholtz operator
\begin{align*}
	\frac{1}{\alpha^2}\nabla^2 -1 &= \frac{1}{\alpha^2}\left(\frac{\partial^2}{\partial x^2}+\frac{\partial^2}{\partial y^2}+\frac{\partial^2}{\partial z^2}\right) - 1 \\
						&= \left(\frac{1}{\alpha^2}\frac{\partial^2}{\partial x^2} - 1\right) \left(\frac{1}{\alpha^2}\frac{\partial^2}{\partial y^2} - 1\right) \left(\frac{1}{\alpha^2}\frac{\partial^2}{\partial z^2} - 1\right) \\
							& \quad +\frac{1}{\alpha^4}\left(\frac{\partial^4}{\partial x^2\partial y^2}+\frac{\partial^4}{\partial x^2\partial z^2}+\frac{\partial^4}{\partial y^2\partial z^2}\right) -\frac{1}{\alpha^6}\frac{\partial^6}{\partial x^2\partial y^2\partial z^2}
\end{align*}
Since $1/\alpha^2 = (c\Delta t)^2/\beta^2$, we have that
\[
	\frac{1}{\alpha^2}\nabla^2-1 = \mathcal{L}_{\beta,x}\mathcal{L}_{\beta,y}\mathcal{L}_{\beta,z} + O((c\Delta t)^4),
\]
where $\mathcal{L}_{\beta,x}$, $\mathcal{L}_{\beta,y}$, and $\mathcal{L}_{\beta,z}$ are one-dimensional modified Helmholtz operators \eqref{eqn:HelmholtzL}, applied in the indicated spatial variable. The local truncation error is fourth order, meaning that the scheme is second order accurate. The corresponding two dimensional operator is
\[
	\frac{1}{\alpha^2}\nabla^2 - 1 = -\mathcal{L}_{\beta,x}\mathcal{L}_{\beta,y} + O((c\Delta t)^4),
\]
and upon omitting the splitting error we arrive at the 2D PDE
\begin{equation}
	\label{eqn:Helmholtz_Equation_2D}
	\mathcal{L}_{\beta,x} [\mathcal{L}_{\beta,y}[ u^{n+1}+u^{n-1}+(\beta^2-2)u^n  ]](x,y) = \beta^2 u^n(x,y) + (c\Delta t)^2 S^n(x,y),
\end{equation}
analogous to the 1D modified Helmholtz equation \eqref{eqn:Helmholtz_Equation}.

\subsection{Boundary corrections in higher dimensions}
The solution of equation \eqref{eqn:Helmholtz_Equation_2D} will require a consistent implementation of boundary conditions, and inclusion of sources. Define the temporary variables $W(x,y)$ and $Z(x,y)$ such that
\begin{align}
	\label{eqn:z_def}
	\mathcal{L}_{\beta,x}[W] =&  \beta^2 u^n + (c\Delta t)^2 S^n, \quad \mathcal{L}_{\beta,y}[Z] =  W.
\end{align}
Thus, the solution $u^{n+1}$ is achieved in two one-dimensional solves, by performing an $x\!-\!y$ sweep
\begin{align}
	\label{eqn:w_Solve}
	W =& \mathcal{L}_{\beta,x}^{-1}[\beta^2 u^n + (c\Delta t)^2 S^n], \\
	\label{eqn:u_Solve}
	u^{n+1}  =& \mathcal{L}_{\beta,y}^{-1}[W]-u^{n-1}-(\beta^2-2)u^{n}.
\end{align}
The intermediate variable $W(x,y)$ is a boundary integral solution, defined by the 1D result \eqref{eqn:L_Inverse}, where $y$ is a treated as a fixed parameter. The homogeneous solution will also exhibit $y-$dependence through the coefficients $A^n = A^n(y)$ and $B^n = B^n(y)$. For a general domain $\Omega$, the lines $y = y_j$ will intersect with the boundary $\partial \Omega$ at different points $x = a_j, b_j$. Thus, equation \eqref{eqn:w_Solve} is solved by applying boundary conditions at $x = a_j$ and $b_j$.

Likewise, for fixed $x$, the $y-$sweep is performed by constructing 1D boundary integral solution $Z$; and finally we solve for $u^{n+1}$ according to \eqref{eqn:u_Solve}. The boundary conditions are applied at $y = c_k, d_k$, which are defined by the intersection of the boundary $\partial \Omega$ with the line $x = x_k$.

Discretization of $\Omega$ is accomplished by embedding it in a regular Cartesian mesh, and additionally incorporating the termination points of the $x$ and $y$ lines, which will always lie on the boundary, and be within $\Delta x$ (or $\Delta y$, respectively) of the nearest grid point inside $\Omega$. For example, the lines and boundary points for a circle are shown in Figure \ref{fig:sweep}. Since the one-dimensional quadrature formula can incorporate unstructured meshes locally without incurring time step restrictions, it is of no concern to have the boundary lie arbitrarily close to a mesh point. That is, we increase the accuracy without affecting stability by including the boundary points. The implementation of boundary conditions requires knowledge of the temporary variable $Z(x,y)$ (and perhaps its derivatives) at the endpoints of each $x$ line. The general approach for implementation of the ADI scheme is as follows:

\begin{enumerate}
	\item
	Lay a mesh over the domain $\Omega$. For each horizontal line $y=y_j$, identify the boundary points $x = a_j$ and $b_j$, where $a_j$ and $b_j$ are the points of intersection of $y = y_j$ with the boundary $\partial \Omega.$ Likewise, identify for each vertical line $x = x_k$ the boundary points $y = c_k$ and $d_k$, which are the intersection of $x = x_k$ with the boundary $\partial \Omega.$ A line object is then identified as all regularly spaced points along an $x$ ($y$) line, including it's boundary points. Assume the number of $x$ and $y$ lines are $N_x$ and $N_y$ respectively.
	\item
	First perform the $x$-sweep. At each time step, construct the temporary variable $W = W(x,y_j)$ defined by \eqref{eqn:w_Solve}, for $1\leq j \leq N_y$. The boundary conditions are imposed at $x = a_j$ and $b_j$.
	 \item
	 Next, perform the $y$-sweep. For $1\leq k \leq N_x$, solve for the variable $Z = u^{n+1}+2u^n+(\beta^2-2)u^{n-1}$, according to \eqref{eqn:u_Solve}. The boundary conditions are now applied at $y = c_k$ and $d_k$.
	\item
	In order to improve the accuracy of the ADI solve, the inversion of the $x$ and $y$ Helmholtz operators is symmetrized, by averaging the results of $x\!-\!y$ and $y\!-\!x$ solves.
\end{enumerate}
The same approach is followed in three dimensions, where the lines are defined by fixing two variables, and similarly finding the endpoints which lie on the boundary $\partial \Omega$ along each line. Symmetrization of the ADI sweeps becomes more tedious, as six different orderings of the ADI sweeps should be taken, and averaged accordingly. However this process still requires $O(N)$ operations.
\begin{figure}[htbp]
	\centering
	\subfigure[$x$ sweep]{\label{fig:sweep-a} \includegraphics[width=.45\textwidth]{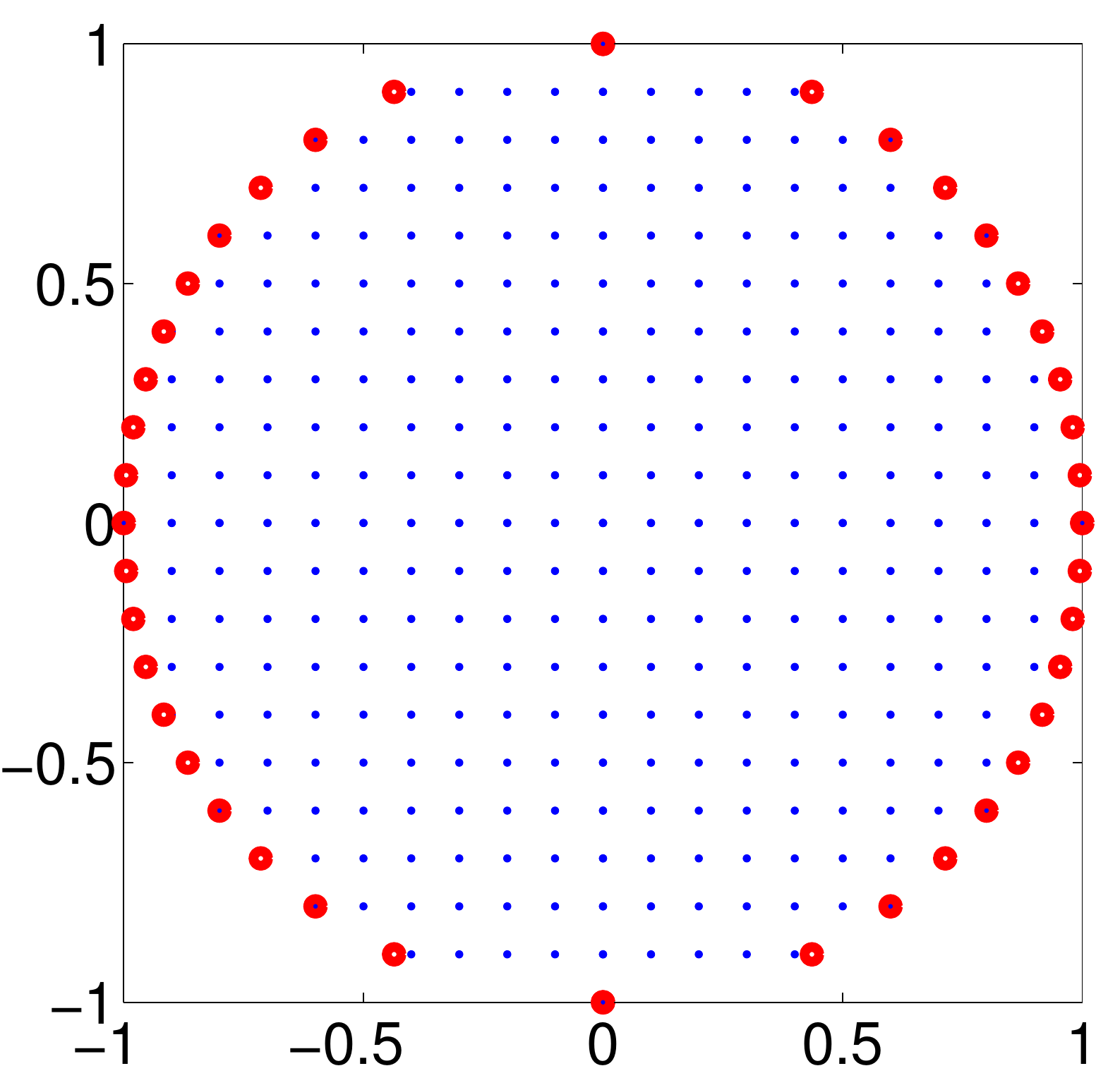}}
	\subfigure[$y$ sweep]{\label{fig:sweep-b}\includegraphics[width=.45\textwidth]{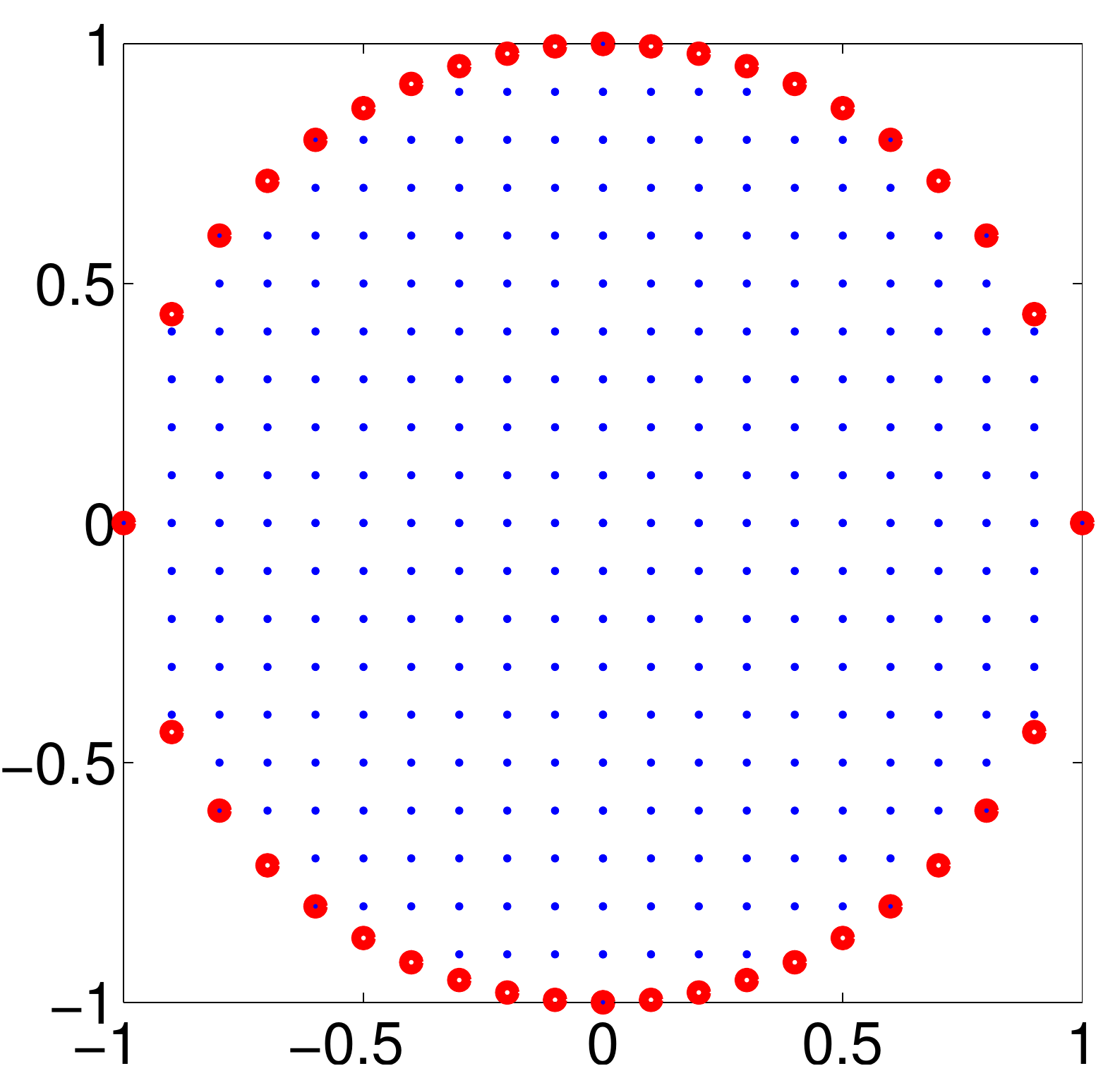}} 
	\caption{Grids used for the ADI $x$ (a) and $y$ (b) sweeps defined in equations \eqref{eqn:w_Solve} and \eqref{eqn:u_Solve} respectively. The red dots are the intersections of the mesh lines with the boundary, and need not be equally spaced.}
	\label{fig:sweep}%
\end{figure}

\section{Numerical Results}
\label{sec:results}

We now demonstrate the ability of our solver to address problems of interest. We first demonstrate the expected second order convergence in both space and time of the algorithm, including the implementation of outflow boundary conditions, and domain decomposition. We next illustrate several two dimensional examples, in non-Cartesian geometries.

\subsection{A one-dimensional example}
We first study the errors produced in 1D by our numerical scheme, as well as those independently produced by our domain decomposition algorithm, and outflow boundary condition implementation. To do this, we construct three separate numerical solutions with the same initial conditions, but with or without domain decomposition, and with or without outflow boundary conditions. The procedure is as follows.

A spatial domain $\Omega = [a,b] \in \mathbb{R}$ is first decomposed into 4 subdomains $\Omega_m = [a_m,b_m],$ for $m=0,1,2,3$. A compactly supported pulse is propagated through the domain, utilizing domain decomposition and imposing outflow boundary conditions, up to time $T = (b-a)/c$, for which the pulse is sure to have left the domain. We also compute independently the numerical solution over $\Omega$, without domain decomposition (but still imposing outflow boundary conditions), so that the difference between the two can be used to measure the error due solely to domain decomposition. A third solution is also constructed, this time over the larger domain $\Omega_{ext} = [a-cT, b+cT]$, that uses neither the domain decomposition nor outflow boundary conditions. The third solution differs from the second only by the numerical reflections caused by outflow for $x \in \Omega$. So in this way, we can independently measure the total discretization errors, 

Since the exact solution is known, we can measure the total discretization error, as well as the numerical reflections due independently to the domain decomposition, and outflow boundary conditions.

The initial pulse we use is a Gaussian, given by
\[
	u(x,0) = \exp\left(-36\left(\frac{2x-b-a}{b-a}\right)^2\right), \quad \text{and} \quad u_t(x,0) = 0.
\]
The subdomains are discretized using regularly spaced mesh points, as well as Chebyshev mesh points, and also using a varied number of spatial points, so that
\[
	x_j = \begin{cases}
	a_0 + (b_0-a_0)\frac{j}{N}, 					\quad &x_j \in \Omega_0 \\
	a_1 + (b_1-a_1)\cos\left(\frac{j\pi}{2N}\right), 		\quad &x_j \in \Omega_1 \\
	a_2 + (b_2-a_2)\frac{j}{2N}, 					\quad &x_j \in \Omega_2 \\
	a_3 + (b_3-a_3)\cos\left(\frac{j\pi}{N}\right), 		\quad &x_j \in \Omega_3.
	\end{cases}
\]
Thus, the full domain $\Omega$ is discretized with a total of $6N+1$ unique spatial points, where $N$ is the number of spatial points in $\Omega_0$. The results are displayed in Tables \ref{tab:refinement_1D_1} - \ref{tab:refinement_1D_3}, for increasing $N$, holding the CFL number fixed. For all simulations, $a = -1$, $b = 1$ the propagation speed is $c = 1$, and the final time is $T = 2$. The CFL number is determined using the maximum mesh spacing on the grid, which for $N = 20$, corresponds to $\Delta x = 0.025$. The CFL numbers are set to $0.1$, $1$ and $10$ in Tables \ref{tab:refinement_1D_1}, \ref{tab:refinement_1D_2} and \ref{tab:refinement_1D_3}, respectively. Second order convergence is observed in each case.

\begin{table}[htbp]
\begin{centering}
\begin{tabular}{|c||c|c||c|c||c|c|}
\hline 
		& \multicolumn{2}{c||}{Domain Decomposition}	& \multicolumn{2}{c||}{Outflow }				& \multicolumn{2}{c|}{Total Discretization}			\\ \hline
$N$		& 	$L^{2}$ error			& $L^{2}$ order& 	$L^{2}$ error			& $L^{2}$ order& 	$L^{2}$ error			& $L^{2}$ order	\\ \hline
$20$		& $\num{0.00479840211}$	& $-$		& $\num{0.00125949765}$	& $-$		& $\num{0.00576216479}$	& $-$	 		\\ \hline
$40$		& $\num{0.00071827454}$	& $2.7399$	& $\num{0.00028492261}$	& $2.1442$	& $\num{0.00140166844}$	& $2.0395$		\\ \hline
$80$		& $\num{0.00009441879}$	& $2.9274$	& $\num{0.00006883611}$	& $2.0493$	& $\num{0.00034557239}$	& $2.0201$		\\ \hline
$120$	& $\num{0.00001220944}$	& $2.9511$	& $\num{0.00001704921}$	& $2.0135$	&$\num{0.00008587556}$	& $2.0087$		\\ \hline
$240$	& $\num{0.00000164072}$	& $2.8956$	& $\num{0.00000425213}$	& $2.0034$	& $\num{0.00002141292}$	& $2.0038$		\\ \hline
$480$	& $\num{0.00000025192}$	& $2.7033$	& $\num{0.00000106236}$	& $2.0009$	& $\num{0.00000534691}$	& $2.0017$		\\ \hline
\end{tabular}
\caption{Refinement study for a 1D Gaussian pulse, demonstrating the convergence of domain decomposition, and outflow boundary conditions. Note that $N$ is the number of spatial points in the first subdomain $\Omega_0$, so the total number of discretization points is $6N+1$. The CFL number is held fixed at $0.1$.}
\label{tab:refinement_1D_1}
\end{centering}
\end{table}

\begin{table}[htbp]
\begin{centering}
\begin{tabular}{|c||c|c||c|c||c|c|}
\hline 
		& \multicolumn{2}{c||}{Domain Decomposition}	& \multicolumn{2}{c||}{Outflow }				& \multicolumn{2}{c|}{Total Discretization}			\\ \hline
$N$		& 	$L^{2}$ error			& $L^{2}$ order& 	$L^{2}$ error			& $L^{2}$ order& 	$L^{2}$ error			& $L^{2}$ order	\\ \hline
$20$		& $\num{0.00197225255}$	& $-$		& $\num{0.00082802594}$	& $-$		& $\num{0.01351638748}$	& $-$ 			\\ \hline
$40$		& $\num{0.00027765931}$	& $2.8285$	& $\num{0.00021952383}$	& $1.9153$	& $\num{0.00347936511}$	& $1.9578$		\\ \hline
$80$		& $\num{0.00003941848}$	& $2.8164$	& $\num{0.00005718893}$	& $1.9406$	& $\num{0.00087845203}$	& $1.9858$		\\ \hline
$120$	& $\num{0.00000594997}$	& $2.7279$	& $\num{0.00001437569}$	& $1.9921$	& $\num{0.00022054626}$	& $1.9939$		\\ \hline
$240$	& $\num{0.00000101772}$	& $2.5475$	& $\num{0.00000359667}$	& $1.9989$	& $\num{0.00005524280}$	& $1.9972$		\\ \hline
$480$	& $\num{0.00000020255}$	& $2.3290$	& $\num{0.00000089927}$	& $1.9998$	& $\num{0.00001382338}$	& $1.9987$		\\ \hline
\end{tabular}
\caption{Results are the same as for Table \ref{tab:refinement_1D_1}, but with the CFL number fixed at 1.}
\label{tab:refinement_1D_2}
\end{centering}
\end{table}

\begin{table}[htbp]
\begin{centering}
\begin{tabular}{|c||c|c||c|c||c|c|}
\hline 
		& \multicolumn{2}{c||}{Domain Decomposition}	& \multicolumn{2}{c||}{Outflow }				& \multicolumn{2}{c|}{Total Discretization}			\\ \hline
$N$		& 	$L^{2}$ error			& $L^{2}$ order& 	$L^{2}$ error			& $L^{2}$ order& 	$L^{2}$ error			& $L^{2}$ order	\\ \hline
$20$		& $\num{0.00103867308}$	& $-$		& $\num{0.02023754735}$	& $-$		& $ 2.4948 \times 10^{-1}$	& $-$			\\ \hline
$40$		& $\num{0.00003081871}$	& $5.0748$	& $\num{0.01730526450}$	& $0.2258$	& $ 1.4802 \times 10^{-1}$	& $0.7531$		\\ \hline
$80$		& $\num{0.00000180867}$	& $4.0908$	& $\num{0.00483219724}$	& $1.8405$	& $ 5.8893 \times 10^{-2}$	& $1.3296$		\\ \hline
$120$	& $\num{0.00000013251}$	& $3.7707$	& $\num{0.00121302111}$	& $1.9941$	& $\num{0.01720171253}$	& $17756$		\\ \hline
$240$	& $\num{0.00000001894}$	& $2.8066$	& $\num{0.00038659365}$	& $1.6497$	& $\num{0.00448533081}$	& $1.9393$		\\ \hline
$480$	& $\num{0.00000000371}$	& $2.3518$	& $\num{0.00010247849}$	& $1.9155$	& $\num{0.00113647811}$	& $1.9806$		\\ \hline
\end{tabular}
\caption{Results are the same as for Table \ref{tab:refinement_1D_1}, but with the CFL number fixed at 10.}
\label{tab:refinement_1D_3}
\end{centering}
\end{table}

\subsection{Two-dimensional examples}

\subsubsection{Rectangular Cavity With Domain Decomposition}

In this section, we demonstrate the second order convergence of our proposed method, including domain decomposition, for a simple rectangular cavity problem with homogeneous Dirichlet and Neumann boundary conditions. A rectangular domain $\Omega = [-L_{x}/2,L_{x}/2]\times[-L_{y}/2,L_{y}/2]$ is divided into four subdomains with new artificial boundaries along $x=0$ and $y=0$, as in Figure \ref{fig:RC_DD}. Due to the Cartesian geometry of this example, the domain decomposition algorithm we use follows directly from the 1D algorithm we have presented. A more general approach will be required on complex subdomains is, the subject of future investigation.

\begin{figure}[H]
\centering
\begin{tikzpicture}[scale=2,>=stealth]

\draw[->] (-0.8,-0.8) -- (-0.6,-0.8)
	node[right] {$x$};
\draw[->] (-0.8,-0.8) -- (-0.8,-0.6)
	node[above]{$y$};

\draw (-1,-1)--(-1,1)--(1,1)--(1,-1)--cycle;

\draw[dashed] (-1,0)--(1,0);
\draw[dashed] (0,-1)--(0,1);

\draw (1,0) node[anchor=west] {$y=0$};
\draw (0,1) node[anchor=south] {$x=0$};

\end{tikzpicture}
\caption{Rectangular cavity with domain decomposition.}
\label{fig:RC_DD}
\end{figure}
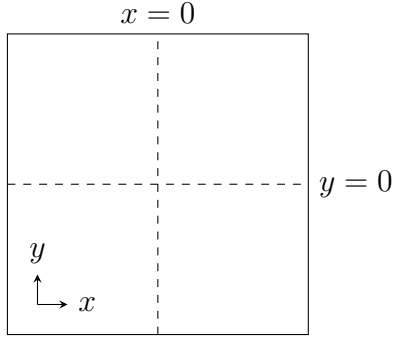

As initial conditions, we choose
\[
	u\left(x,y,0\right)  = 
	\begin{cases}
	\cos\left(\frac{(2m+1)\pi x}{L_{x}}\right)\cos\left(\frac{(2n+1)\pi y}{L_{y}}\right)	& \text{Dirichlet case} \\
	\sin\left(\frac{(2m+1)\pi x}{L_{x}}\right)\sin\left(\frac{(2n+1)\pi y}{L_{y}}\right)	& \text{Neumann case}
\end{cases}
\]
and
\[
	u_{t}\left(x,y,0\right)=0
\]
for $(x,y) \in \Omega$, $m$ and $n$ integers. Exact solutions are well-known in each case. The results of refinement studies are listed in Tables \ref{tab:refinement_RCD} and \ref{tab:refinement_RCN}. The error is the maximum  discrete $L^{2}$ error (computed against the exact solution) over all time steps.

\begin{table}[htbp]
\begin{centering}
\begin{tabular}{|c||c|c||c|c||c|c|}
\hline 
 		& \multicolumn{2}{c||}{CFL 0.5} 			& \multicolumn{2}{c||}{CFL 2} 				& \multicolumn{2}{c|}{CFL 10}			 \\ \hline
$\Delta x$	& $L^{2}$ error 			& $L^{2}$ order& $L^{2}$ error 		& $L^{2}$ order	& $L^{2}$ error & $L^{2}$ order		\\ \hline
$1/40$ 	& $\num{0.0008480082}$ 	& $-$		& $\num{0.00589191}$	& $-$			& $\num{0.106922697}$ & $-$			\\ \hline
$1/80$ 	& $\num{0.0001901801}$ 	& $2.15671$	& $\num{0.00145775}$	& $2.01499$		& $\num{0.033923646}$ & $1.65620$	\\ \hline
$1/160$ 	& $\num{0.0000448109}$ 	& $2.08544$	& $\num{0.00036061}$	& $2.01523$		& $\num{0.008868315}$ & $1.93555$	\\ \hline
$1/320$ 	& $\num{0.0000108611}$ 	& $2.04468$	& $\num{0.00008955}$	& $2.00965$		& $\num{0.002207988}$ & $2.00592$	\\ \hline
$1/640$ 	& $\num{0.0000026726}$ 	& $2.02284$ 	& $\num{0.00002230}$	& $2.00535$		& $\num{0.000546378}$ & $2.01476$ 	\\ \hline
\end{tabular}
\caption{Refinement study for rectangular cavity with Dirichlet BC using domain decomposition. Here, $c=1$, $m=n=0$, and $L_{x}=L_{y}=1$.}
\label{tab:refinement_RCD}
\end{centering}
\end{table}

\begin{table}[htbp]
\begin{centering}
\begin{tabular}{|c||c|c||c|c||c|c|} \hline 
		& \multicolumn{2}{c||}{CFL 0.5} 			& \multicolumn{2}{c||}{CFL 2}				& \multicolumn{2}{c|}{CFL 10}			\\ \hline
$\Delta x$	& $L^{2}$ error 			& $L^{2}$ order& $L^{2}$ error				& $L^{2}$ order& $L^{2}$ error & $L^{2}$ order		\\ \hline
$1/40$	& $\num{0.000889276}$		& $-$		& $\num{0.00618621}$		& $-$		& $\num{0.1114727}$ & $-$			\\ \hline
$1/80$	& $\num{0.000194862}$		& $2.1902$	& $\num{0.00149418}$		& $2.0497$	& $\num{0.0347712}$ & $1.6807$		\\ \hline
$1/160$	& $\num{0.0000453669}$		& $2.1027$	& $\num{0.00036511}$		& $2.0329$	& $\num{0.0089791}$ & $1.9532$		\\ \hline
$1/320$	& $\num{0.0000109286}$		& $2.0535$	& $\num{0.00009011}$		& $2.0185$	& $\num{0.0022217}$ & $2.0148$		\\ \hline
$1/640$	& $\num{0.0000026809}$		& $2.0273$	& $\num{0.00002237}$		& $2.0098$	& $\num{0.0005480}$ & $2.0192$		\\ \hline
\end{tabular}
\caption{Refinement study for rectangular cavity with Neumann BC using domain decomposition. Here, $c=1$, $m=n=0$, and $L_{x}=L_{y}=1$.}
\label{tab:refinement_RCN}
\end{centering}
\end{table}

\subsubsection{Double Circle Cavity}

In this example, we solve the wave equation with homogeneous Dirichlet boundary conditions on a 2D domain $\Omega$ which is, as in Figure \ref{fig:dblcirc_geom}, the union of two overlapping disks, with centers $P_{1}=\left(- \gamma,0\right)$ and $P_{2}=\left(\gamma,0\right)$, respectively, and each with radius $R$:
\begin{equation} \nonumber
\Omega =  \left\{\left(x,y\right) : |\left(x,y\right)-P_{1}| < R \right\}\cup \left\{\left(x,y\right) : |\left(x,y\right)-P_{2}|<R \right\}
\end{equation}
\noindent where $|\left(x,y\right)| = \sqrt{x^{2}+y^{2}}$ is the usual Euclidean vector norm, and $\gamma < R$.

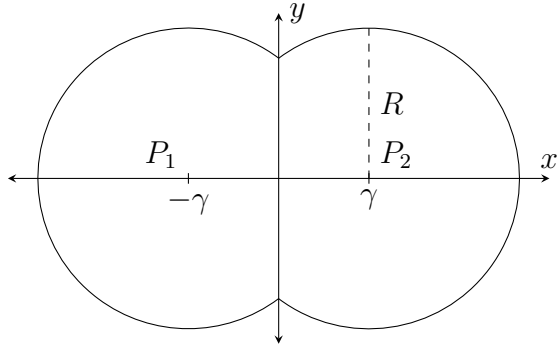
\begin{figure}[H]
\centering
\begin{tikzpicture}[scale=2,>=stealth]

\draw[<->] (-1.8,0) -- (1.8,0)
	node [above] {$x$};
\draw[<->] (0,-1.1) -- (0,1.1)
	node [right] {$y$};
\draw (0,-.8) arc (-126.9:126.9:1);
\draw (0,.8) arc (53.1:306.9:1);
\draw[dashed] (.6,0) -- (.6,1);
\draw (.6,.5) node[anchor=west] {$R$};

\foreach \x in {-.6,.6}
{
	\draw (\x,-1pt)--(\x,1pt);
}
	
\draw (.6,0) node[anchor=north] {$\gamma$};
\draw (-.6,0) node[anchor=north] {$-\gamma$};

\draw (.6,0) node[anchor=south west] {$P_{2}$};
\draw (-.6,0) node[anchor=south east] {$P_{1}$};

\end{tikzpicture}
\caption{Double circle geometry.}
\label{fig:dblcirc_geom}
\end{figure}
This geometry is of interest due to, for example, its similarity to that of the radio frequency (RF) cavities used in the design of linear particle accelerators, and presents numerical difficulties due to the curvature of, and presence of corners in, the boundary. Our method avoids the stair-step approximation used in typical finite difference methods to handle curved boundaries, which reduces accuracy to first order and may introduce spurious numerical diffraction.

\begin{figure}[ht!]
	\centering
	\subfigure[$t=0          $]{\includegraphics[trim={1cm 1cm 1cm 1cm},clip, width=0.3\textwidth]{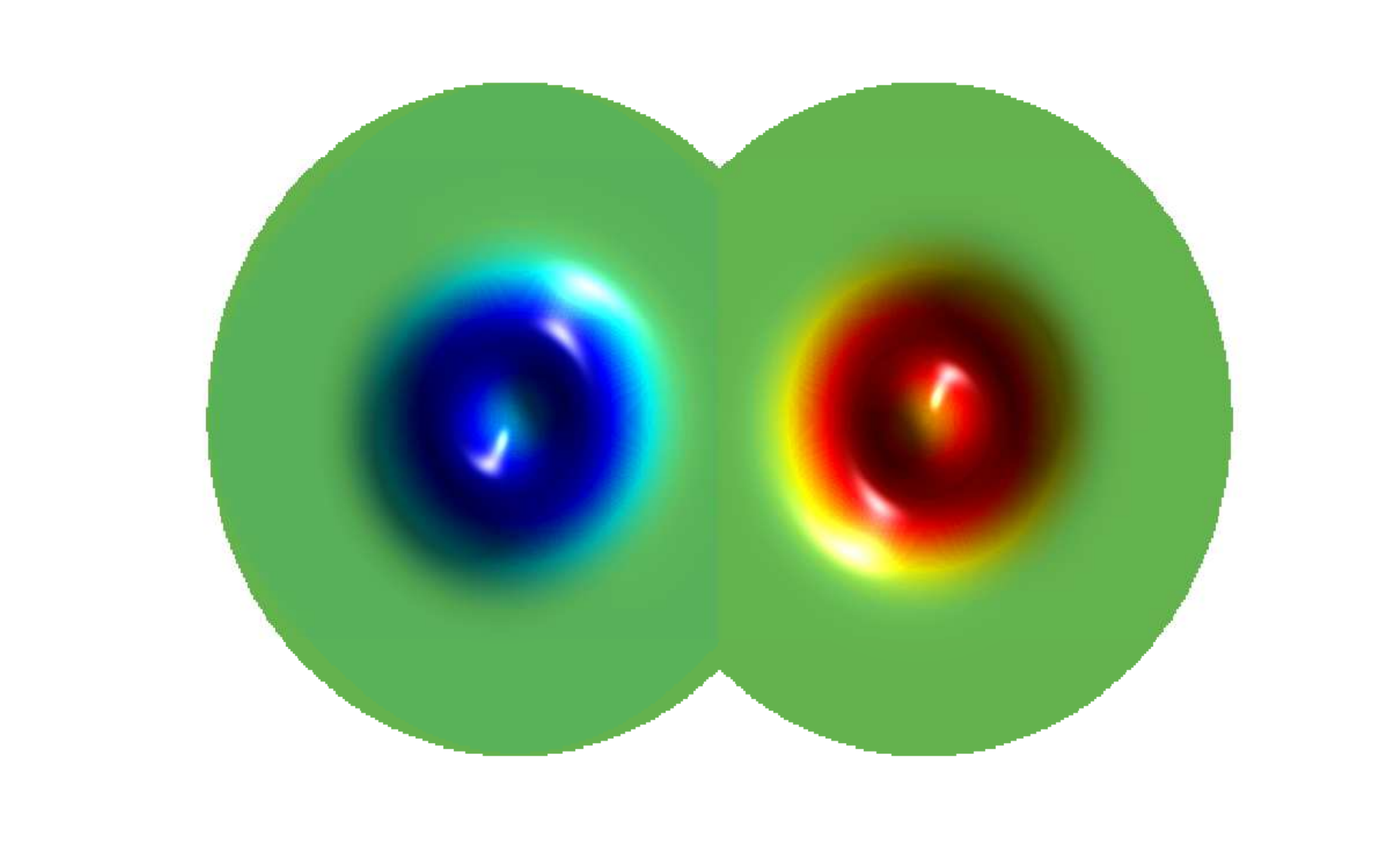}}
	\subfigure[$t=0.5145$]{\includegraphics[trim={1cm 1cm 1cm 1cm},clip, width=0.3\textwidth]{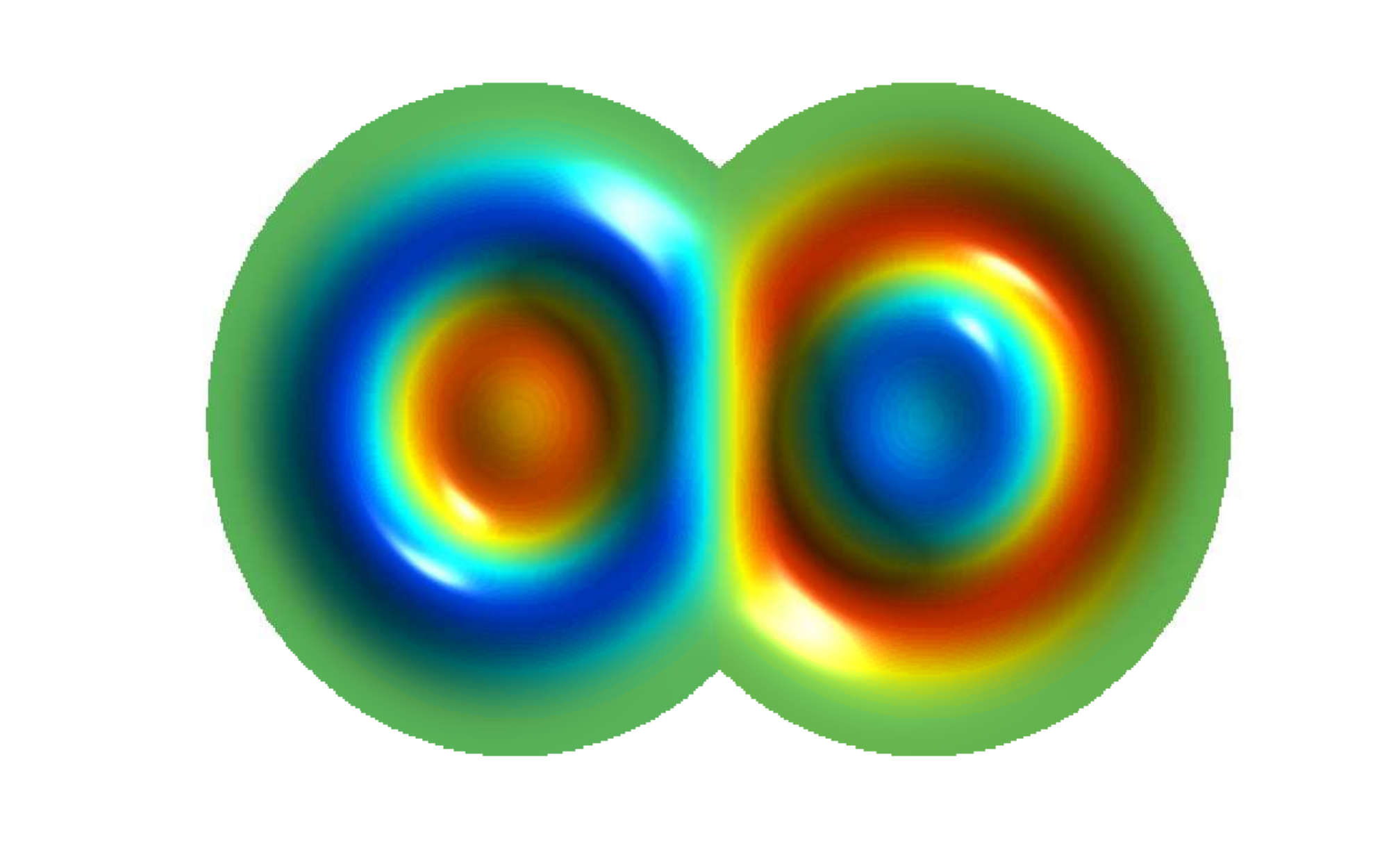}}
	\subfigure[$t=0.5145$]{\includegraphics[trim={1cm 1cm 1cm 1cm},clip, width=0.3\textwidth]{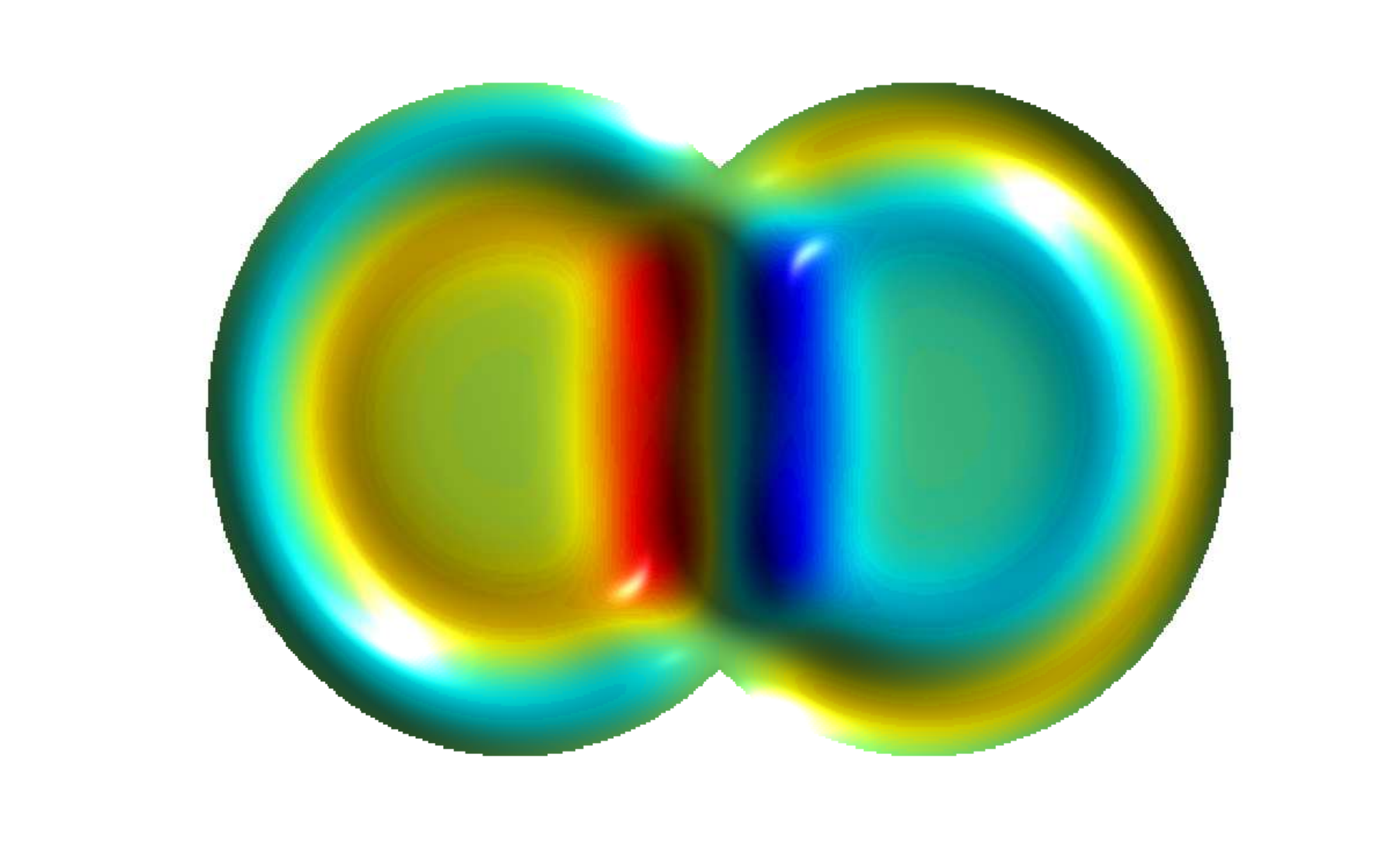}}
	\subfigure[$t=0.5145$]{\includegraphics[trim={1cm 1cm 1cm 1cm},clip, width=0.3\textwidth]{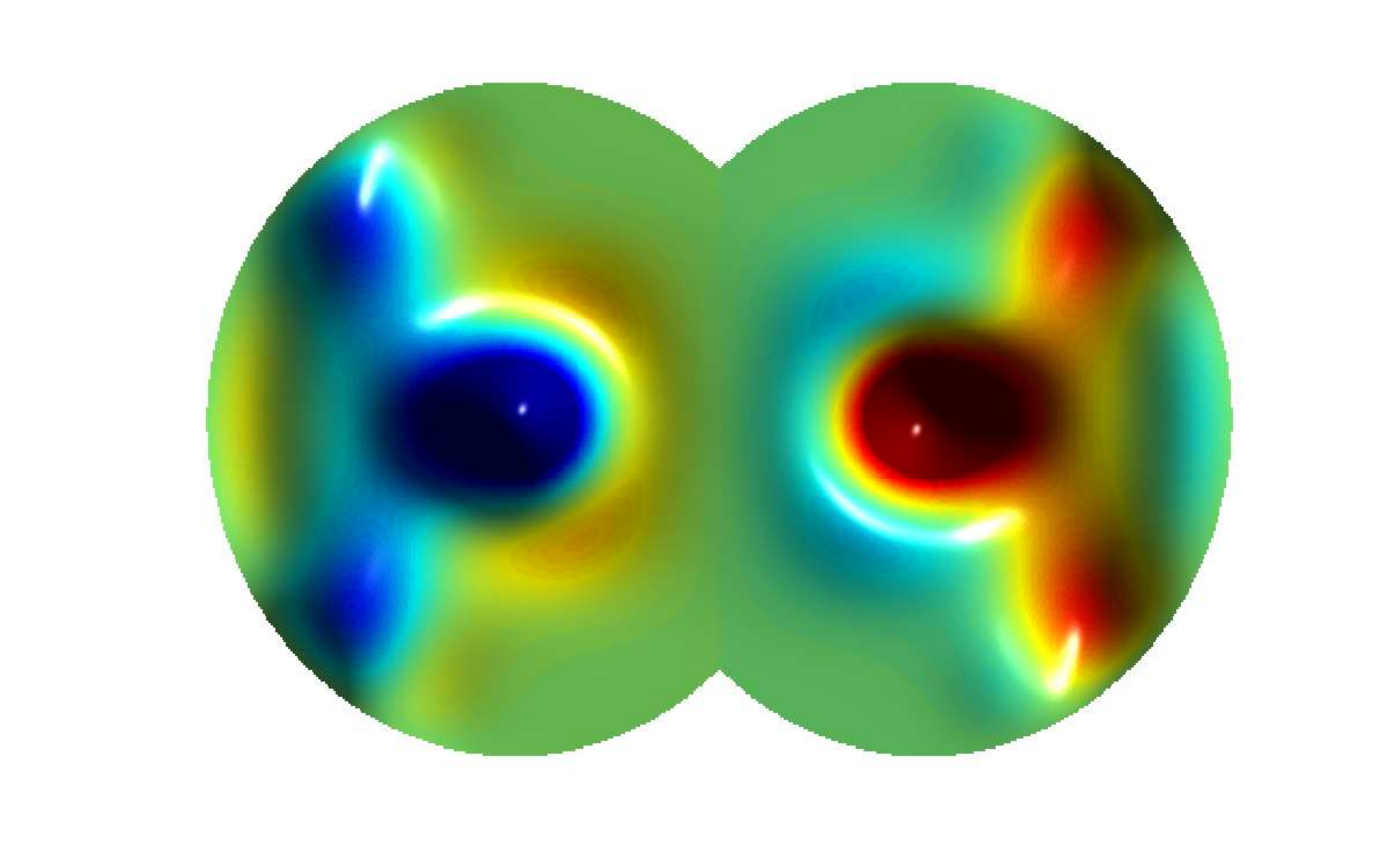}}
	\subfigure[$t=0.9135$]{\includegraphics[trim={1cm 1cm 1cm 1cm},clip, width=0.3\textwidth]{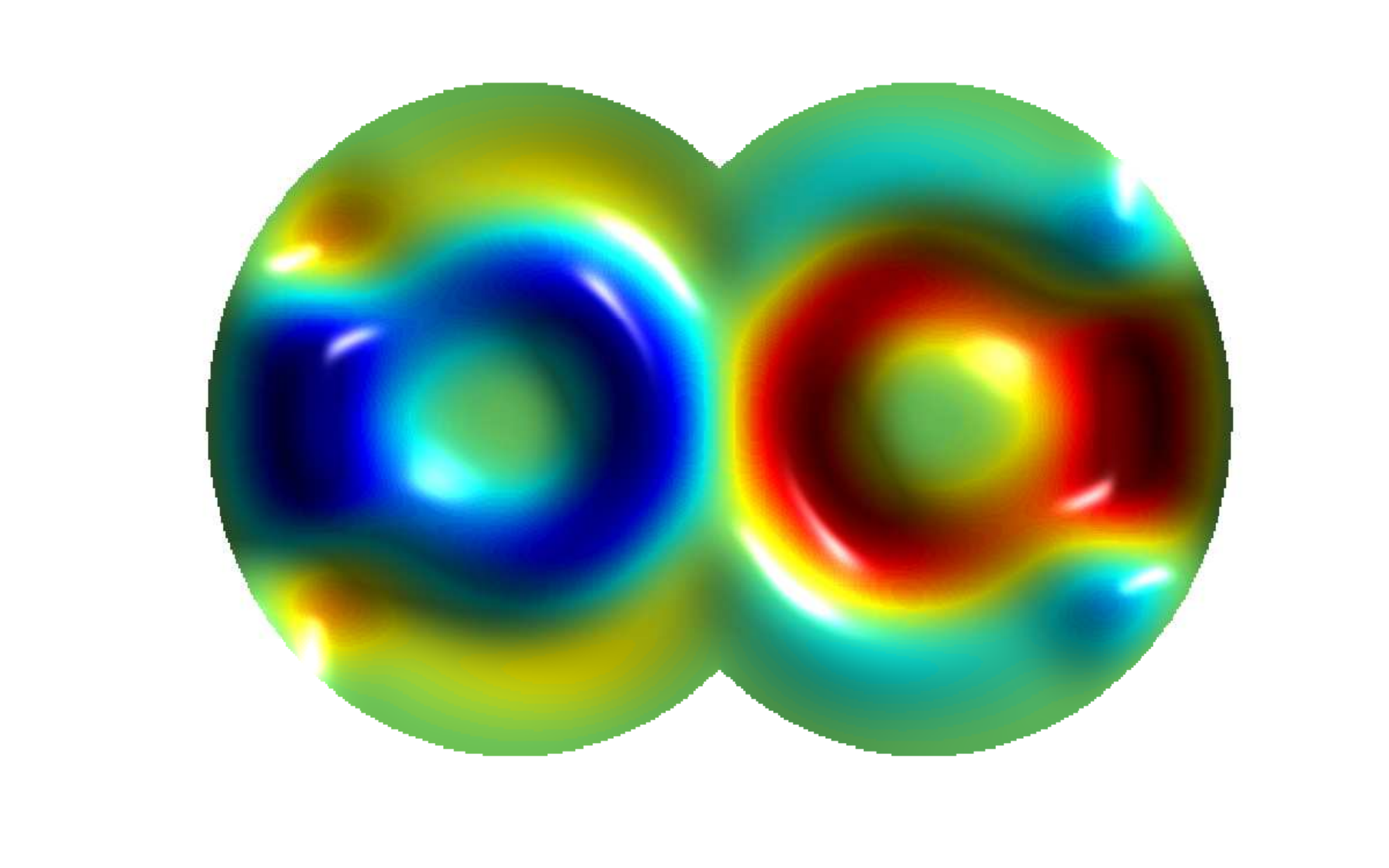}}
	\subfigure[$t=1           $]{\includegraphics[trim={1cm 1cm 1cm 1cm},clip, width=0.3\textwidth]{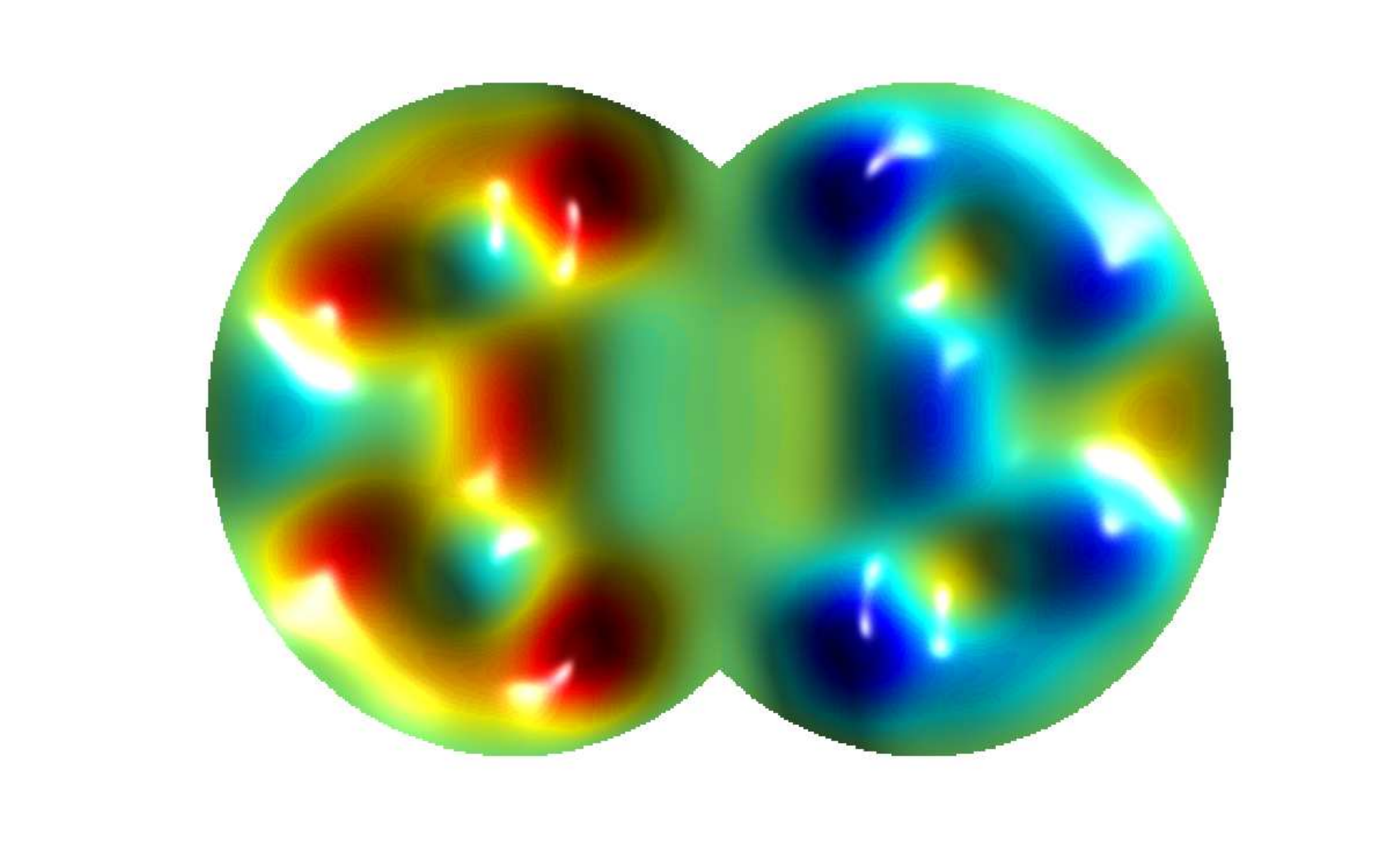}}
	\caption{Evolution of the double circle cavity problem.}
	\label{fig:dblcirc_plots}
\end{figure}

As initial conditions, we choose
\[
	u\left(x,y,0\right)  = 
	\begin{cases}
	-\cos^{6}\left(\frac{\pi}{2}\left(\frac{|(x,y)-P_{1}|}{0.8 \gamma}\right)^{2}\right)	& |\left(x,y\right)-P_{1}| < 0.8 \gamma \\
	\cos^{6}\left(\frac{\pi}{2}\left(\frac{|(x,y)-P_{2}|}{0.8 \gamma}\right)^{2}\right)	& |\left(x,y\right)-P_{2}| < 0.8 \gamma \\
	0 & \text{otherwise}
\end{cases}
\]
and
\[
	u_{t}\left(x,y,0\right)=0
\]
for $(x,y) \in \Omega$. Selected snapshots of the evolution are given in Figure \ref{fig:dblcirc_plots}, and the results of a refinement study are given in Table \ref{tab:refinement_dblcirc}. The discrete $L^{2}$ error was computed against a well-refined numerical reference solution ($\Delta x = \num{0.00021875}$); the error displayed in the table is the maximum over time steps with $t \in [0.28,0.29]$. For this example, $R=0.3$, $\gamma=0.2$, $c=1$, and the CFL is 2.

\begin{table}[htbp]
\begin{centering}
\begin{tabular}{|c|c|c||c|c|} \hline 
$\Delta x$  		& $\Delta y$ 			& $\Delta t$			& $L^{2}$ error			& $L^{2}$ order	\\ \hline
$\num{0.007}$		& $\num{0.0043333333}$	& $\num{0.008666667}$	& $\num{0.006143688}$	& $-$			\\ \hline
$\num{0.0035}$	& $\num{0.0021666667}$	& $\num{0.004333333}$	& $\num{0.001682923}$	& $1.8681$		\\ \hline
$\num{0.00175}$	& $\num{0.0010833333}$	& $\num{0.00216666}$	& $\num{0.000435945}$	& $1.9488$		\\ \hline
$\num{0.000875}$	& $\num{0.0005416667}$	& $\num{0.00108333}$	& $\num{0.000105150}$	& $2.0517$		\\ \hline
\end{tabular}
\caption{Refinement study for the double circle cavity with Dirichlet BC. For the numerical reference solution, $\Delta x = \num{.00021875}$, $\Delta y = \num{.00013542}$, and $\Delta t = \num{.00027083}$.}
\label{tab:refinement_dblcirc}
\end{centering}
\end{table}


\subsubsection{Symmetry on a Quarter Circle}

With the goal of testing the capabilities of our boundary conditions as well as circular geometry, we construct standing modes on a circular wave guide of radius $R$, in two different ways. First, we solve the Dirichlet problem, with initial conditions
\[
	u(x,y,0) = J_0\left(z_{20}\frac{r}{R} \right), \quad u_t(x,y,0) = 0,
\]
and exact solution $u = J_0\left(z_{20}\frac{r}{R} \right)\cos\left(z_{20}\frac{ct}{R}\right)$, where $J_0$ is the Bessel function of order 0, and $z_{20} = 5.5218$ is the $2$-nd zero. Secondly, we use the symmetry of the mode to construct the solution restricted to the second quadrant, with homogeneous Neumann boundary conditions taken along the $x$ and $y$ axes.

In both cases, the solution converges to second order. An overlay of the two are shown in Figure \ref{fig:Bessel_Quarter}, demonstrating the close agreement.

\begin{figure}[hb]
\centering
	\subfigure[$t = 0.25$]{\includegraphics[width = 0.24\textwidth]{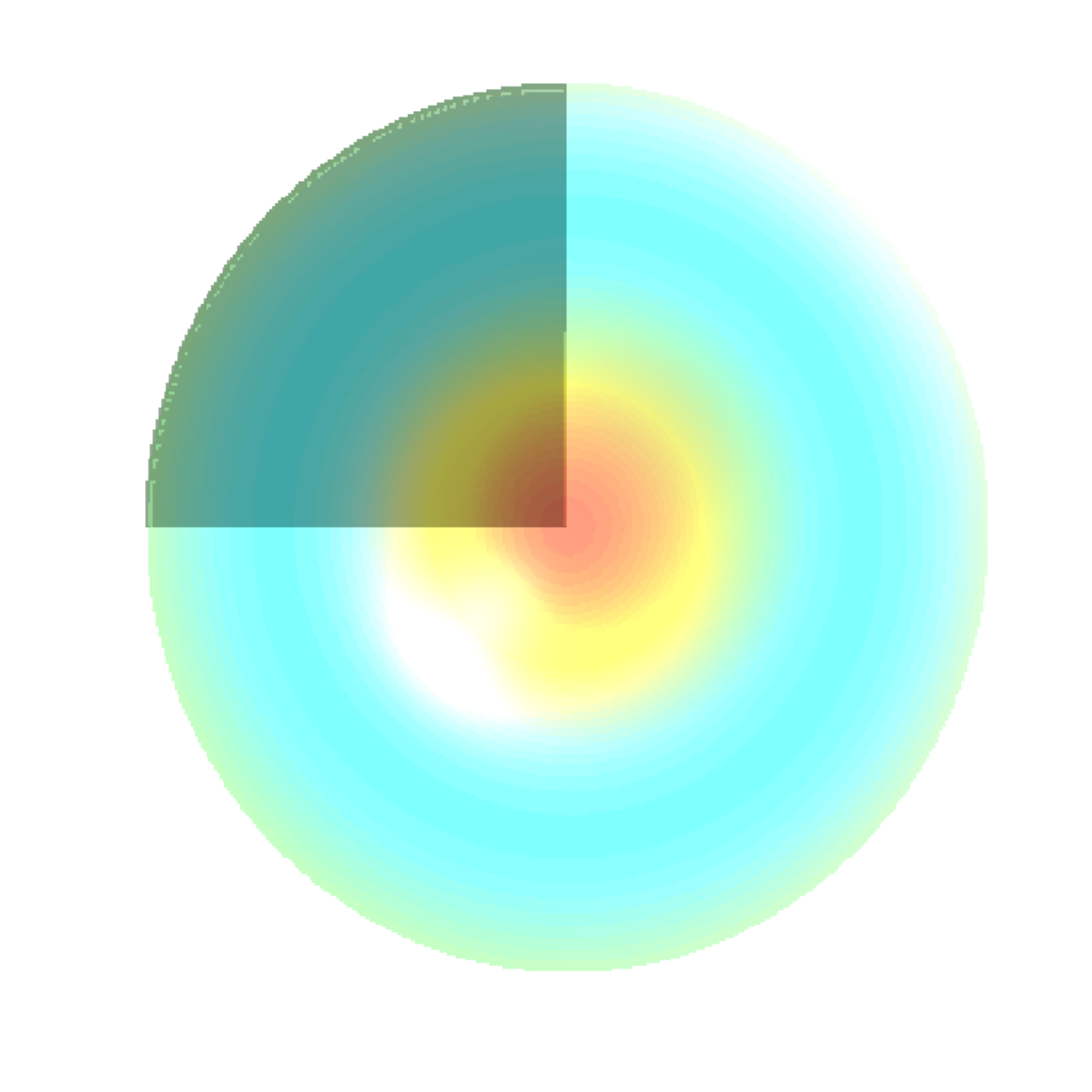}}
	\subfigure[$t = 0.50$]{\includegraphics[width = 0.24\textwidth]{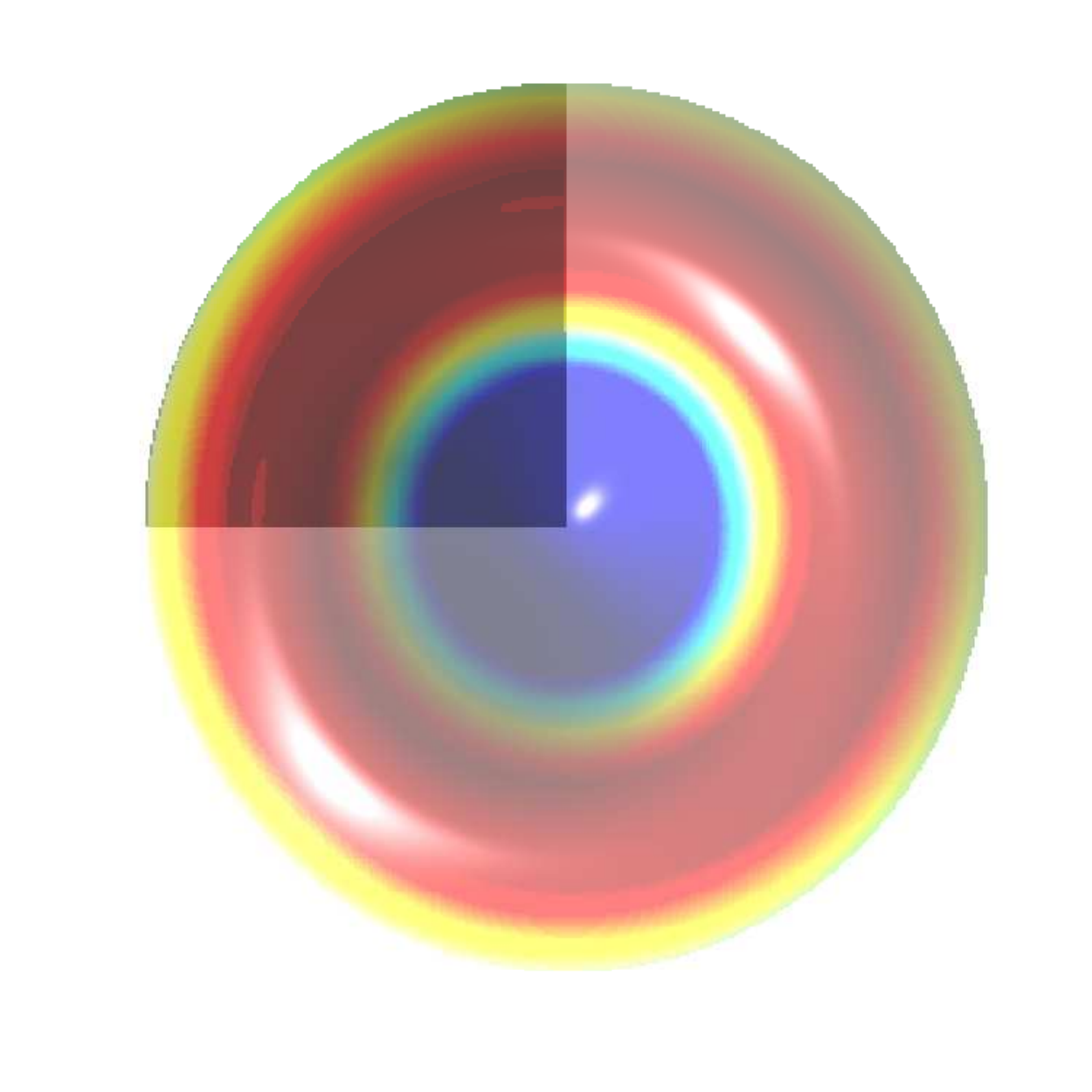}}
	\subfigure[$t = 0.75$]{\includegraphics[width = 0.24\textwidth]{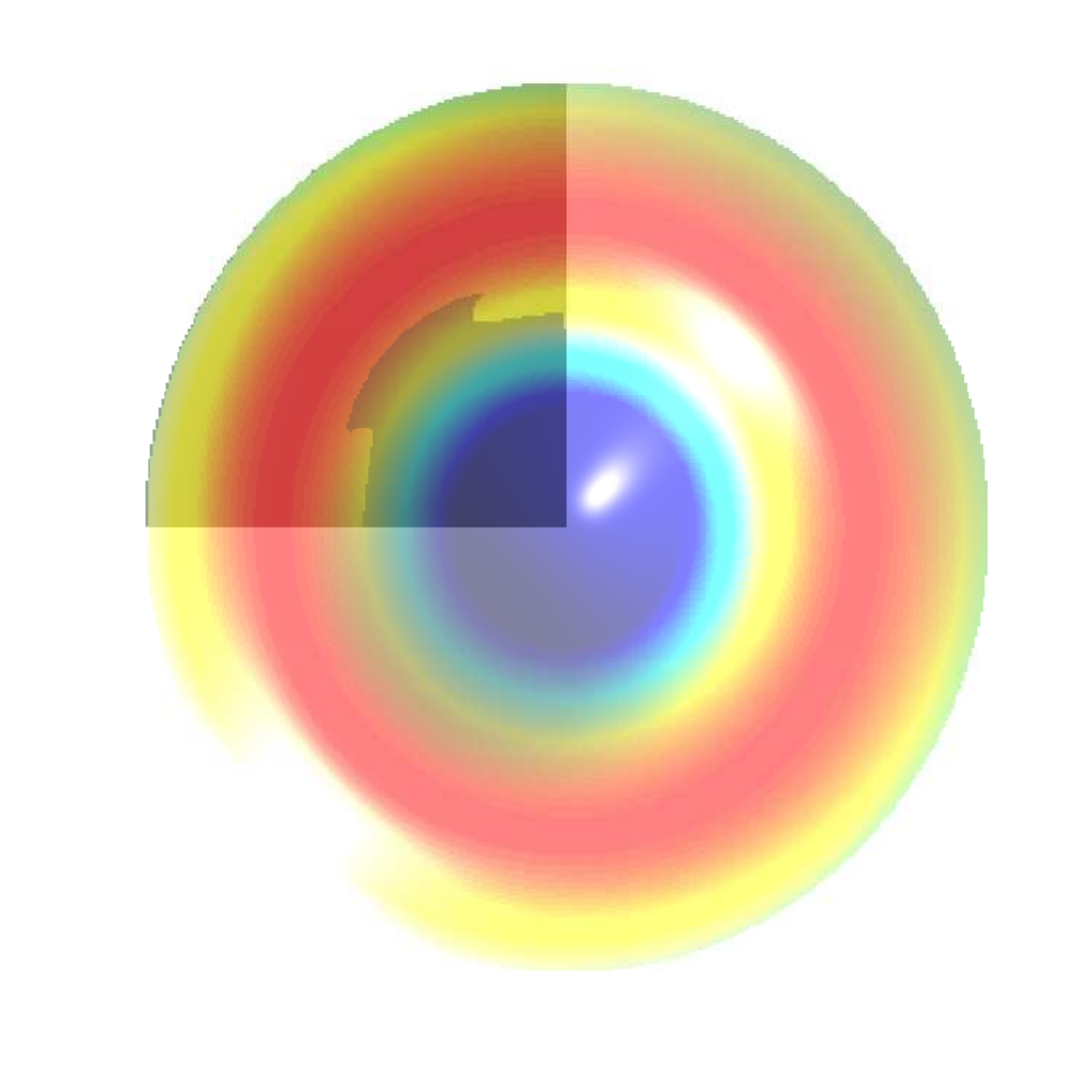}}
	\subfigure[$t = 1.00$]{\includegraphics[width = 0.24\textwidth]{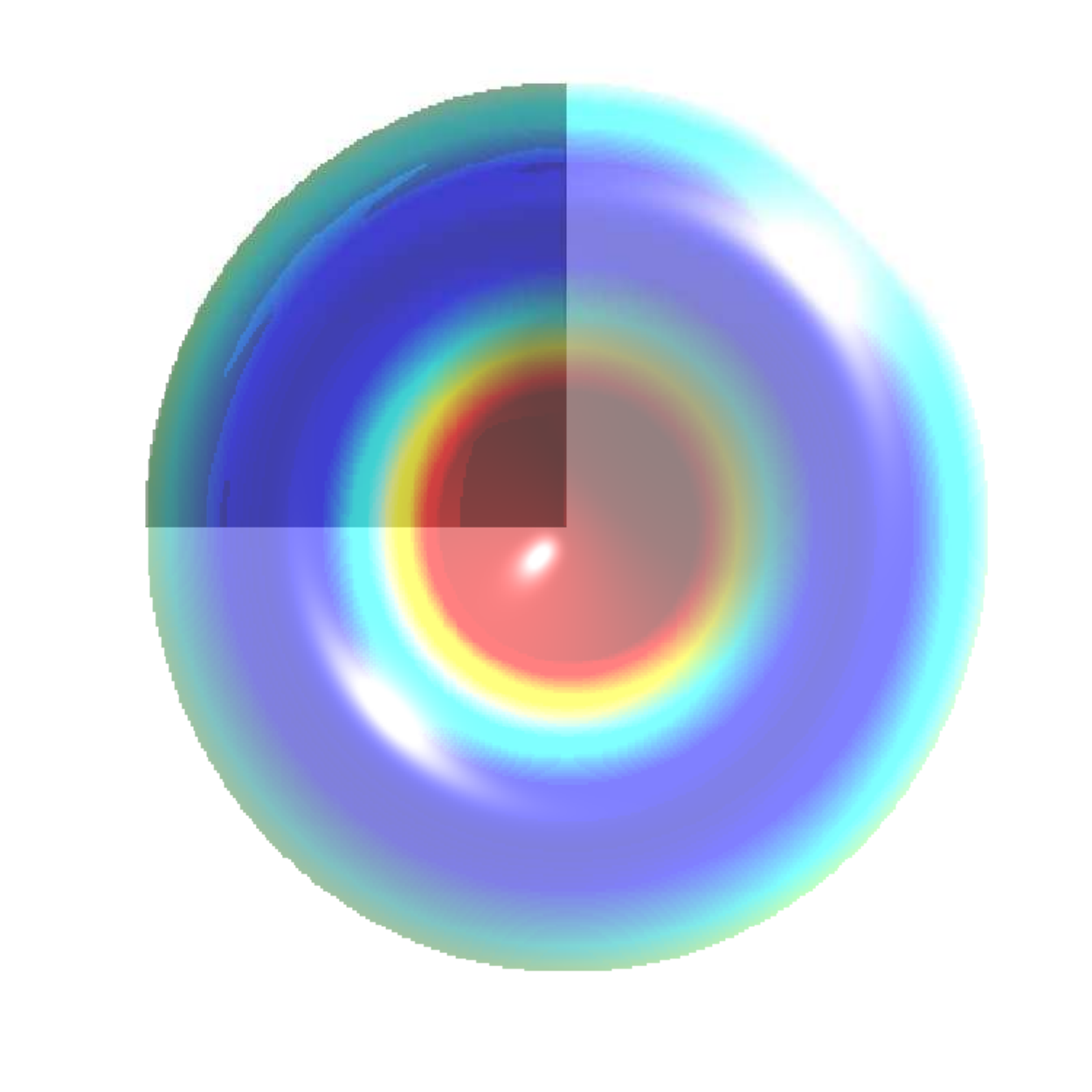}}
	\caption{Two separate numerical constructions of a Bessel mode are superimposed, demonstrating that the solution on the quarter circle using Neumann boundary conditions is equivalent to that of the full circle.}
	\label{fig:Bessel_Quarter}
\end{figure}

\subsubsection{Periodic Slit Diffraction Grating}

Diffraction gratings are periodic structures used in optics to separate different wavelengths of light, much like a prism. The high resolution that can be achieved with diffraction gratings makes them useful in spectroscopy, for example, in the determination of atomic and molecular spectra. In this example, we apply our method to model an infinite, periodic diffraction grating under an incident plane wave. The purpose of this example is to demonstrate the use of our method with multiple boundary conditions and nontrivial geometry in a single simulation to capture complex wave phenomena.

\begin{figure}[H]
\centering
\begin{tikzpicture}[scale=2,>=stealth]

\draw (-1,0)--(-0.2,0);
\draw (0.2,0)--(1,0);

\draw[->] (-0.8,-0.8) -- (-0.6,-0.8)
	node[right] {$x$};
\draw[->] (-0.8,-0.8) -- (-0.8,-0.6)
	node[above]{$y$};

\draw [decorate,decoration={brace,amplitude=5pt},xshift=0pt,yshift=1pt]
(-0.2,0) -- (0.2,0) node [black,midway,xshift=0,yshift=10pt] {$a$};
\draw [decorate,decoration={brace,amplitude=5pt},xshift=0pt,yshift=12pt]
(-1,0) -- (1,0) node [black,midway,xshift=0,yshift=15pt] {$d$};

\draw[->,decorate,decoration={snake,amplitude=2,post length=.08}] 		
	(0,-.7) -- (0,-.2)
	node [right,align=center,midway]
	{
	{$u_{\mbox{\scriptsize inc}}$}
	};

\draw[dashed] (-1,-1)--(-1,1)--(1,1)--(1,-1)--cycle;

\draw[anchor=north] (0,-1) node {Outflow BC};
\draw[anchor=south] (0,1) node {Outflow BC};
\draw[anchor=west] (1,0) node {Periodic BC};
\draw[anchor=east] (-1,0) node {Periodic BC};

\end{tikzpicture}
\caption{Periodic slit diffraction grating geometry}
\label{fig:slit_geom}
\end{figure}
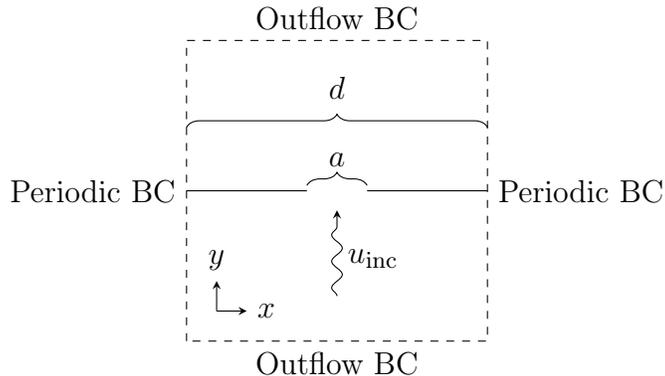


In the next example, we perform a preliminary test of outflow boundary conditions in higher dimensions. While a rigorous analysis of the algorithm is the subject of future work, the results look quite reasonable. Our numerical experiment is depicted in Figure \ref{fig:slit_geom}. An idealized slit diffraction grating consists of a reflecting screen of vanishing thickness, with open slits of aperture width $a$, spaced distance $d$ apart, measured from the end of one slit to the beginning of the next (that is, the periodicity of the grating is $d$).

\begin{figure}[htbp!]
	\centering
\subfigure[$t=0.31$]{\includegraphics[trim={1cm 1cm 1cm 1cm},clip, width = .24\textwidth]{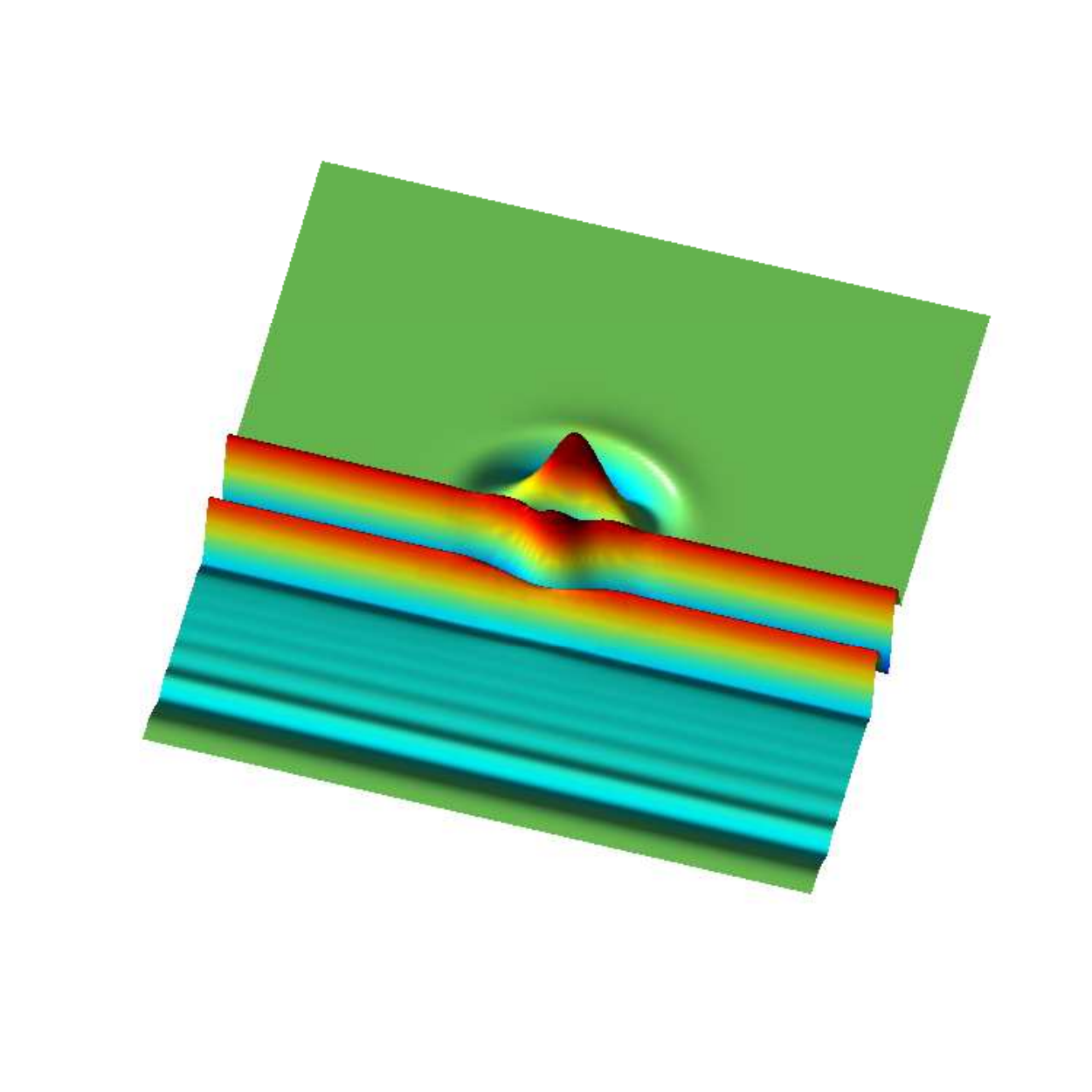}}
\subfigure[$t=0.51$]{\includegraphics[trim={1cm 1cm 1cm 1cm},clip, width = .24\textwidth]{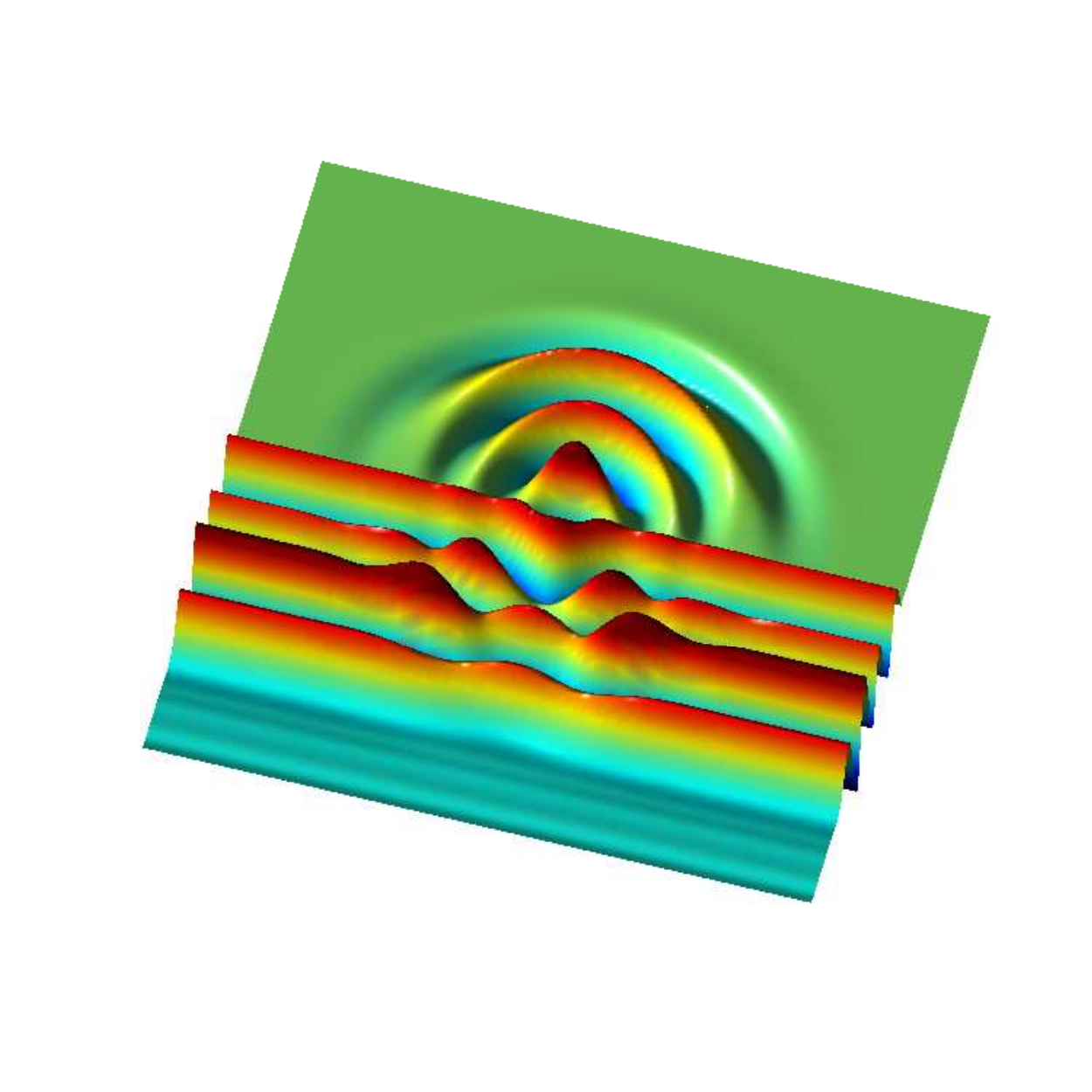}}
\subfigure[$t=1.01$]{\includegraphics[trim={1cm 1cm 1cm 1cm},clip, width = .24\textwidth]{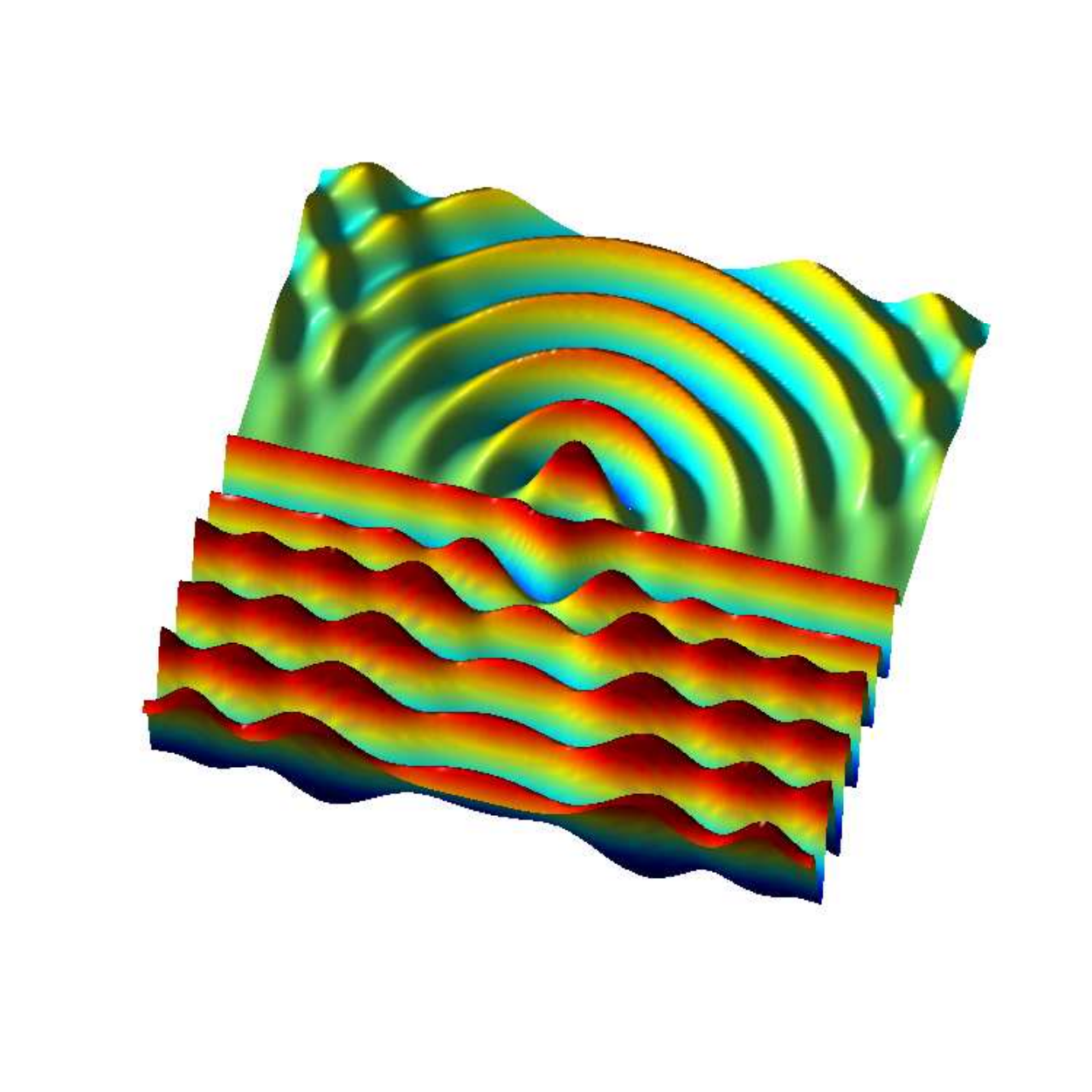}}
\subfigure[$t=2.01$]{\includegraphics[trim={1cm 1cm 1cm 1cm},clip, width = .24\textwidth]{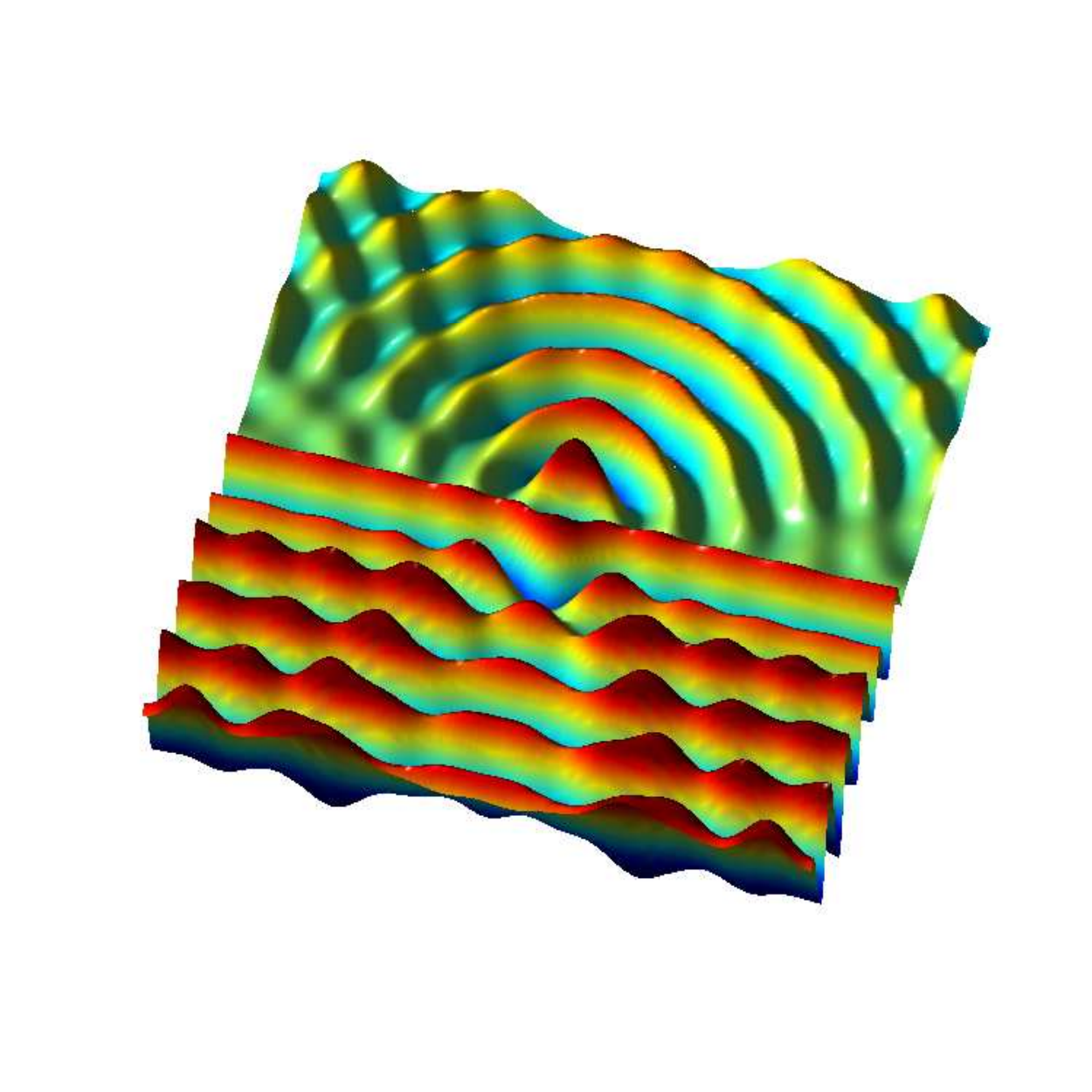}}
	\caption{Evolution of the slit diffraction grating problem, with aperture width $a=0.1$, grating periodicity $L_{y}=d=1$, and wave speed $c=1$. The CFL is fixed at 2.}
	\label{fig:slit_plots}
\end{figure}

We impose an incident plane wave of the form $u_{inc}(x,y,t)=\cos{\left(\omega t + ky\right)}$, where $k = 2 \pi / a$ and $\omega = k/c$, where $c$ is the wave speed. We impose periodic BCs at $x=\pm d/2$ (determining the periodicity of the grating), and homogeneous Dirichlet BCs on the screen. The outflow boundary conditions are imposed at $y=\pm L_{y}/2$. In Figure \ref{fig:slit_plots}, we observe the time evolution of the incident plane wave passing through the aperture, and the resulting interference patterns as the diffracted wave propagates across the periodic boundaries. The outflow boundary conditions allow the waves to propagate outside the domain, with no visible reflections at the artificial boundaries.

\section{Conclusion}
\label{sec:conclusion}

In this paper we have presented a fast, A-stable, second order method for solving the wave equation. Using the Method of Lines Transpose (MOL$^T$), we formulate the semi-discrete problem in one spatial dimension, and solve the resulting boundary value problem using an $O(N)$ fast convolution algorithm. From the the underlying exponential recurrence relation, upon which our algorithm is based, we also develop a means to employ domain decomposition, as well as formulate outflow boundary conditions. Additionally, we address the inclusion of point (delta function) sources into our solver, which is of interest in simulations which couple wave propagation with particle dynamics, i.e. plasma simulations.

We have also extended our fast algorithm to higher spatial dimensions using alternate direction implicit (ADI) splitting, and demonstrated second order convergence of our solver in non-Cartesian geometries, with DIrichlet, Neumann, periodic  and outflow boundary conditions. While our efforts to employ domain decomposition in 2D have only consisted of regular Cartesian subdomains, our results are very promising, and this will be investigated in future work.

\bibliographystyle{amsplain}
\bibliography{MOLT_ADI}

\end{document}